\documentclass{amsart}
\usepackage[utf8]{inputenc}
\usepackage{amsthm}
\usepackage{amsrefs}
\usepackage{amssymb}
\usepackage{mathtools}
\usepackage{amsmath}
\usepackage{mathrsfs}
\usepackage{xcolor}
\usepackage{hyperref}
\usepackage{calligra}
\usepackage[T1]{fontenc}
\usepackage{tikz-cd}

\theoremstyle{plain}
\newtheorem{theorem}{Theorem}
\newtheorem{proposition}{Proposition}

\newtheorem{lemma}[theorem]{Lemma}
\theoremstyle{definition}
\newtheorem{definition}{Definition}
\theoremstyle{remark}
\newtheorem{remark}{Remark}

\newcommand{\Z}{\mathbb{Z}}
\newcommand{\Q}{\mathbb{Q}}
\newcommand{\N}{\mathbb{N}}
\newcommand{\F}{\mathbb{F}}
\newcommand{\OO}{\mathcal{O}}
\newcommand{\Gal}{\mathrm{Gal}}
\newcommand{\Cl}{\mathrm{Cl}}
\newcommand{\p}{\mathfrak{P}}
\newcommand{\G}{\mathscr{G}}
 \DeclareFontFamily{U}{wncy}{}
    \DeclareFontShape{U}{wncy}{m}{n}{<->wncyr10}{}
    \DeclareSymbolFont{mcy}{U}{wncy}{m}{n}
    \DeclareMathSymbol{\Sh}{\mathord}{mcy}{"58}

\DeclareMathOperator{\Ima}{Im}
\DeclareMathOperator{\Nrm}{Nm}
\DeclareMathOperator{\Hom}{Hom}
\DeclareMathOperator{\Iw}{Iw}
\DeclareMathOperator{\CoInd}{CoInd}
\begin{document}
\title{Iwasawa $\lambda$ invariant and Massey product}
\author{Peikai Qi}
\date{\today}
\address{Michigan State University, East Lansing, Michigan, USA}
\email{qipeikai@msu.edu}
\thanks{Thanks to my advisor Preston Wake for guiding me in learning Iwasawa theory. He suggested the question to me and motivated me to do the research. Thanks for the support of the Department of Mathematics at Michigan State University.}

\begin{abstract}

We compute Iwasawa $\lambda$ invariant in terms of Massey products in Galois cohomology with restricted ramification. When applied to imaginary quadratic fields and cyclotomic fields, we obtain a new proof and generalization of results of Gold \cite{Gold} and McCallum-Sharifi \cite{MR2019977}. The main tool is the generalized Bockstein map introduced by Lam-Liu-Sharifi-Wake-Wang\cite{MR4537772}.
\end{abstract}
\maketitle

\section{Introduction}

\subsection{Background}
Let $K$ be a number field and $K\subset K_1\subset K_2\subset\cdots\subset K_l \subset\cdots K_\infty$ be a $\Z_p$ extension of $K$. Let $X=\varprojlim \Cl(K_l)[p^\infty]$, where $\Cl(K_l)[p^\infty]$ denotes the $p$-part of the class group of $\Cl(K_l)$. Let $\mu$ and $\lambda$ be the Iwasawa invariants of $X$. Our goal is to relate the value of $\lambda$ with the vanishing of Massey products under the assumption that $\mu=0$.

Massey products of Galois cohomology are introduced in number theory to study the structure of Galois groups. We will introduce the definition of Massey products in section \ref{Massey product sect}. One can view Massey products as a generalization of cup products in cohomology and for example, the cup product is a $2$-fold Massey product.

The idea of using Massey products to study Iwasawa theory first appears in Sharifi's paper \cite{MR2312552}. McCallum and Sharifi proved that under some assumptions, we have $\lambda\geq 2$ if and only if a certain cup product vanishes for cyclotomic fields in ~\cite{MR2019977}*{Proposition 4.2}. One can also translate Gold's criterion \cite{Gold} into group cohomology. It also has the form that $\lambda\geq 2$ if and only if a certain cup product vanishes for imaginary quadratic fields under some assumptions. The two results have completely different proof. We want to find the deep reason behind it. Our main theorem unifies these two results and when applied to these cases, we get a generalization for them. Roughly speaking, we proved that under some assumptions, if $\lambda\geq n-1$, then $\lambda \geq n$ if and only if a certain Massey product vanishes.

For readers who are familiar with Massey products, here are some differences one should notice. In J\'an Min\'a\v c and Nguy\~{\^e}n Duy T\^ an's Massey products vanishing conjecture \cite{MR3584563}, Massey products vanishing means that Massey products vanish relative to all defining systems. In our paper, Massey products vanishing means that Massey products vanish relative to a particular defining system, which we call the proper defining system. In addition, they consider the Massey products for absolute Galois groups and we consider the Galois groups with restricted ramifications.

\subsection{The strategy and notations}
The strategy of the project is divided into four steps.

\begin{enumerate}
    \item \textit{The Iwasawa invariant $\lambda$ and $\lambda_{cs}$ }

    Let $S$ be the set of primes of $K$ above $p$ and $X_{cs}=\varprojlim \Cl_S(K_l)[p^\infty]$. Let $\mu_{cs}$ and $\lambda_{cs}$ be the Iwasawa invariant for the Iwasawa module $X_{cs}$. Let $D_l$ be the subgroup of $\Cl(K_l)[p^\infty]$ generated by primes in $S$. Then we have 
    \[
    0\rightarrow \varprojlim D_l\rightarrow X\rightarrow  X_{cs}\rightarrow 0
    \]
    We can relate $\lambda$ with $\lambda_{cs}$ if we know $\varprojlim D_l$.
    \item  \textit{The size of $H^2(G_{K_l,S},\mu_p)$ and $\lambda_{cs}$ }

    Let $K_S$ be the maximal extension of $K$ unramified outside $S$ and $G_{K_l,S}=\Gal(K_S/K_l)$. 
    We have the following exact sequences from Kummer theory:
\[
   0\rightarrow \Cl_S(K_l)/p  \rightarrow H^2(G_{K_l,S},\mu_p)\rightarrow Br(\OO_{K_l}[1/p])[p]\rightarrow 0 
\]
To know the information about $\lambda_{cs}$, we need information about the size of the group $\Cl_S(K_l)[p^\infty]$. Hence, we need information about the size of $H^2(G_{K_l,S},\mu_p)$.
     \item  \textit{The generalized Bockstein map and the size of $H^2(G_{K_l,S},\mu_p)$ }

     By Lam-Liu-Sharifi-Wake-Wang's paper \cite{MR4537772}*{Theorem 2.2.4.}, We have the following formula:
      \[\frac{I^nH_{\Iw}^2G_{K_l,S},\mu_p)}{I^{n+1}H_{\Iw}^2(G_{K_l,S},\mu_p)}\cong \frac{H^2(G_{K,S},\mu_p)}{\Ima\Psi^{(n)}}\]
where $I$ is augmentation ideal of $\F_p[[\Gal(K_l/K)]]$ and $\Psi^{(n)}: H^1(G_{K,S},\mu_p\otimes I^n/I^{n+1})\rightarrow H^2(G_{K,S},\mu_p)$ is the generalized Bockstein map defined in \cite{MR4537772}. We have a filtration $H_{\Iw}^2(G_{K_l,S},\mu_p)\supset I H_{\Iw}^2(G_{K_l,S},\mu_p)\supset I^2H_{\Iw}^2(N,\mu_p)\supset \cdots \supset I^n H_{\Iw}^2(G_{K_l,S},\mu_p)\cdots$. Once we know the size of $H^2(G_{K,S},\mu_p)$ and the size of $\Ima\Psi^{(n)}$, we can determine the filtration and get the information about the size of $H_{\Iw}^2(G_{K_l,S},\mu_p) $.
      \item  \textit{ The generalized Bockstein map and Massey products}

In Lam-Liu-Sharifi-Wake-Wang's paper \cite{MR4537772}, they proved that under some circumstances, the image of generalized Bockstein map $\Ima\Psi^{(n)}$ is spanned by certain $n$-fold Massey products.
    
\end{enumerate}

\subsection{Main theorem and corollaries}
As one can see, our strategy does not involve the Iwasawa Main Conjecture. And the Iwasawa $\lambda$ invariant that we computed is the algebraic Iwasawa $\lambda$ invariant. By following the strategy, one of the main theorems is the following:

\begin{theorem} \label{theorem 1}
    Let $K\subset K_1\subset K_2\subset \cdots \subset K_\infty$ be a $\Z_p$ extension of $K$ and $S$ be the set of primes above $p$ for $K$. Assume all primes in $S$ are totally ramified in $K_\infty/K$. Let $X_{cs}=\varprojlim\Cl_S(K_l)$ and $\mu_{cs}$, $\lambda_{cs}$  be the Iwasawa invariants of $X_{cs} $. Assume $X_{cs} $ has no torsion element and $H^2(G_{K,S},\mu_p)\cong \F_p$. 
    
    Then $\mu_{cs}=0$ if and only if there exists integer $k$ such that $\Psi^{(k)}\neq 0$ for some $k$. If $\mu_{cs}=0$, then 
    \[
    \lambda_{cs}=\min\{n|\Psi^{(n)}\neq 0\}-\#S+1
    \]
\end{theorem}

When we have a Galois group $\Delta$  acts on the Iwasawa module  $X=\varprojlim \Cl(K_n)$ and the action gives us a decomposition $X=\oplus_i \varepsilon_i X$, the techniques used to prove Theorem \ref{theorem 1} also apply equivalently to calculate the Iwasawa invariant $\lambda_i$ of $ \varepsilon_i X$. See Theorem \ref{action main theom}.

We applied the Theorem \ref{theorem 1} in many cases in the paper. In the introduction, we list two cases that correspond to the setting of results of Gold \cite{Gold} and McCallum-Sharifi \cite{MR2019977}. 
 
\begin{theorem}
Let $K$ be an imaginary quadratic field and assume that  $p\nmid h_K$, $p$ splits in $K$ as $p\OO_K=\p_0\tilde{\p}_0$. For cyclotomic $\Z_p$ extension, the $\lambda$-invariants of $K$ can be determined in terms of Massey products as follows:

Let $n\geq 2$ and suppose $\lambda\geq n-1$. Then $\lambda\geq n$ if and only if $n$-fold Massey product $(\chi,\chi,\cdots\chi,\alpha)$ is zero with respect to the proper defining system.

Here $\chi$ is a character $\chi: G_{K,S}\rightarrow \Gal(K_\infty/K)\cong\Z_p$ and $\alpha$ is the generator of the principal idea $\p_0^{h_K}$.
\end{theorem}
\begin{remark}
    In other words, it means
    \[
    \lambda=\min\{n \mid n \text{ fold Massey products } (\chi,\chi,\cdots\chi,\alpha) \text{ is nonzero}\}
    \]
\end{remark}
\begin{remark}
It is a fact that $\lambda\geq 1$ in the case. Gold's criterion \cite{Gold} said that $\lambda\geq 2\Leftrightarrow \alpha^{p-1}\equiv 1 \mod{\title{\p}_0^2}$. Further calculation shows that $\alpha^{p-1}\equiv 1 \mod{\title{\p}_0^2}\Leftrightarrow \log_p(\alpha)\equiv0\mod{p^2}\Leftrightarrow \chi\cup \alpha=0$, where $\log_p$ is the $p$-adic logarithm. The theorem can be viewed as new proof of Gold's result and a generalization of Gold's result.

\end{remark}

\begin{theorem}
   Let $K=\Q(\mu_p)$ and $\omega: \Gal(\Q(\mu_p)/\Q)\cong (\Z/p\Z)^*\rightarrow \Z_p$ be the Teichm\"{u}ller character. We can decompose the class group $\Cl(K)[p^\infty]$ as 
   \[
  \Cl(K)[p^\infty] = \oplus_{i=0}^{p-2}\varepsilon_i \Cl(K)[p^\infty]\] where $ \varepsilon_i=\frac{1}{p-1}\sum_{a=1}^{p-1}\omega^i(a)\sigma_a^{-1}\in \Z_p[\Gal(\Q(\mu_p)/\Q]$. Let $\lambda_i$ be the $\lambda$ invariant corresponding to $\varepsilon_i \Cl(K_l)[p^\infty]$.

   Fix $i=3,5,\cdots p-2$ and assume that $\varepsilon_i \Cl(K)[p^\infty]$ is cyclic.  Let $n\geq 2$ and suppose $\lambda_i\geq n-1$, then $\lambda_i\geq n$ if and only if $n$-fold Massey product $\varepsilon_i(\chi,\chi,\cdots\chi,\alpha_i)=0$ with respect to the proper defining system.

\end{theorem}
\begin{remark}
   The assumption $\varepsilon_i\Cl(K)[p]=\F_p$ implies that $\lambda_i\geq 1$ in the case. Note that $\varepsilon_i\Cl(K)[p^\infty]$ is cyclic if Vandiver's conjecture holds.  The theorem implies that $\lambda_i\geq 2\Leftrightarrow \varepsilon_i(\chi,\alpha_i)=0\Leftrightarrow\chi\cup \alpha_i=0$. Proposition 4.2 in McCallum and Sharifi's paper \cite{MR2019977} describes a similar result that $\lambda_i\geq 2\Leftrightarrow \chi\cup \alpha_i=0$. The theorem can be viewed as a generalization of McCallum and Sharifi's results. 
\end{remark}
\subsection{Structure of the paper}
In section \ref{Generalized Bockstein map and Massey product}, we discuss the generalized Bockstein map introduced by Lam-Liu-Sharifi-Wake-Wang\cite{MR4537772} and recall the relation of generalized Bockestein map and Massey products. We will also prove that the generalized Bockstein map preserves the group action. In section \ref{Size of $H^2$}, we prove a formula to determine the size of the second cohomology group by the generalized Bockstein map. In section \ref{Applications to number theory}, we prove the main theorem by applying the formula into number theory. We also list four cases in which we apply our theorem to get some interesting results. The first three cases are cyclotomic $\Z_p$ extensions of imaginary quadratic fields. The last case is the cyclotomic $\Z_p$ extension of cyclotomic fields. In the first case, we also develop a numerical criterion to determine $\lambda$, which takes 10 pages. One can skip the numerical criterion for the first reading.

\section{Generalized Bockstein map and Massey product} \label{Generalized Bockstein map and Massey product}
\subsection{Generalized Bockstein map} \label{generalized Bockstein map}
In this section, we recall the definition and properties of the generalized Bockstein map from \cite{MR4537772}.

Let $G$ be a profinite group of finite $p$-cohomological dimension $d$ and $N$ be a closed normal subgroup such that $G/N$ is a finitely generated pro-$p$ quotient. In the paper, we take $G/N\cong \Z_p$ or $\Z/p^l\Z$. However, the definition of generalized Bockstein map works more generally. Let $\Omega =\F_p[[G/N]]$ be the completed group algebra which is a $G$-module in a natural way. Let $\sigma$ be the generator of $G/N$ and $I=<\sigma -1>$  be the augmentation ideal in $\Omega$. For $0\leq n<\#G/N$, we have the following exact sequence of $G-$module:
\[0\rightarrow I^n/I^{n+1} \rightarrow \Omega/I^{n+1}\rightarrow \Omega/I^n\rightarrow 0\]

After taking tensor product with a finite $\F_p[G]-$module $T$, it is still an exact sequence since every module that appeared above is a $\F_p$ module:

\[0\rightarrow  I^n/I^{n+1} \otimes_{\F_p} T\rightarrow \Omega/I^{n+1}\otimes_{\F_p} T\rightarrow \Omega/I^n\otimes_{\F_p} T\rightarrow 0\]
For $0\leq n< \#G/N$, 
define the generalized Bockstein map $\Psi^{(n)}$ to be the connecting map
\[\Psi^{(n)}: H^{d-1}(G, \Omega/I^n\otimes T)\rightarrow H^d(G, I^n/I^{n+1}\otimes T)\cong H^d(G,T)\otimes  I^n/I^{n+1} \]
where the last isomorphism uses the fact that $I^n/I^{n+1}$ is a trivial $G$-module.
We view $\Psi^{(0)}=0$.
Recall the definition of Iwasawa cohomology groups:
\[
H^r_{\Iw}(N,T)=\varprojlim_{N\leq U \trianglelefteq^{\circ} G  } H^r(U,T)
\] where the inverse limit is taken with respect to correstriction maps and $U$ runs over all open normal subgroups of $G$ containing $N$. Notice that if $G/N$ is finite, then $H^i_{\Iw}(N,T)=H^i(N,T)$. 
In Lam-Liu-Sharifi-Wake-Wang's paper, they proved:
\begin{theorem}[Theorem A in LLSWW\cite{MR4537772}]
\label{filtration}
    For each $0\leq n< \#G/N$, there is a canonical isomorphism
    \[ \frac{I^nH_{\Iw}^d(N,T)}{I^{n+1}H_{\Iw}^d(N,T)}\cong \frac{H^d(G,T)\otimes I^n/I^{n+1}}{\Ima\Psi^{(n)}}\]
    of $\F_p$-modules, where $d$ is the $p$-cohomological dimension of $G$.
\end{theorem}

For the remaining part of this subsection, we will show that the isomorphism above has a certain equivalence.  We can decompose the cohomological group into a direct sum of eigenspaces with respect to a group action. For each eigenspace, we still have such isomorphism. However, the process of checking that the generalized Bockstein map $\Psi^{(n)}$ preserves the group action is tedious. One can skip the part to Remark \ref{skip1}.  

\begin{lemma}\label{action1}
    Let $G/N\cong\Z_p$ and $U_l$ be the unique open normal subgroup of $G$ containing $N$ such that $G/U_l\cong \Z/p^l\Z$. And $T$ is a finite $\F_p[G]-$module.  Then $\CoInd_{U_l}^G T:=\Hom_{\Z U_l}(\Z G, T)\cong \F_p[G/U_l]\otimes_{\F_p} T$ as $\F_p[G] -$ module. By Shapiro's lemma,
    \[H^r(U_l,T)\cong H^r(G,\F_p[G/U_l]\otimes_{\F_p} T)\]
\end{lemma}
\begin{proof}
    Let $G=\bigsqcup_{i=1}^{p^l} U_l \sigma_i$ where $\sigma_i$ are  right coset representatives. Now we define a homomorphism $\alpha: \Hom_{\Z U_l}(\Z G, T)\rightarrow \F_p[G/U_l]\otimes_{\F_p} T$ by mapping the element $\phi\in \Hom_{\Z U_l}(\Z G, T)$ to the element $ \sum_{i=1}^{p^l}\bar{\sigma}_i^{-1}\otimes \sigma_i^{-1} \phi(\sigma_i)$, where $\bar{\sigma}_i$ represents the image of $
    \sigma_i$ in the quotient $ G/U_l$. 
    
    First, the map does not depend on the choice of right coset representatives. Let $ h_i\sigma_i$ be another set of right coset representatives where $h_i\in U_l $. then 
    \[\sum_{i=1}^{p^l}\bar{\sigma}_i^{-1}\bar{h}_i^{-1}\otimes \sigma_i^{-1} h_i^{-1}\phi(h_i\sigma_i)= \sum_{i=1}^{p^l}\bar{\sigma}_i^{-1}\otimes \sigma_i^{-1} \phi(\sigma_i)\]
    since $ \phi(h_i\sigma_i)=h_i\phi(\sigma_i)$. 
    
    Second, the map preserves $\F_p[G]$ actions. Recall the action $g\in G$ on $\phi\in \Hom_{\Z U_l}(\Z G, T)$ is \[(g\psi)(x)=\psi(xg).\] The action $g\in G$ on $ \sum_{i=1}^{p^l}\bar{\sigma}_i\otimes t_i\in \F_p[G/U_l]\otimes_{\F_p} T $ is
    \[g(\sum_{i=1}^{p^l}\bar{\sigma}_i\otimes t_i)=\sum_{i=1}^{p^l}\bar{g}\bar{\sigma}_i\otimes g t_i.\] 
    Then \[\alpha(g\phi)=\sum_{i=1}^{p^l}\bar{\sigma}_i^{-1}\otimes \sigma_i^{-1} (g\phi)(\sigma_i )=\sum_{i=1}^{p^l}\bar{\sigma}_i^{-1}\otimes \sigma_i^{-1} \phi(\sigma_i g).\]
    Assume $\sigma_i g=h_i \sigma_{\delta(i)}$, where $\delta$ is a permutation of $1\leq i\leq p^l$.  We have 
    \[\sigma_i^{-1} \phi(\sigma_i g)=\sigma_i^{-1} \phi(h_i \sigma_{\delta(i)}) =\sigma_i^{-1} h_i\phi( \sigma_{\delta(i)})=g \sigma_{\delta(i)}^{-1}\phi( \sigma_{\delta(i)}) \]
    and 
    \[\Bar{\sigma}_i^{-1}=(\overline{h_i\sigma_{\delta(i)}g^{-1}})^{-1}=\bar{g}\bar{\sigma}_{\delta(i)}^{-1}.\] 
    Hence 
    \[
\alpha(g\phi)=\sum_{i=1}^{p^l}\bar{g}\bar{\sigma}_{\delta(i)}^{-1}\otimes g \sigma_{\delta(i)}^{-1}\phi( \sigma_{\delta(i)}) =\sum_{i=1}^{p^l}\bar{g}\bar{\sigma}_i^{-1}\otimes g \sigma_i^{-1} \phi(\sigma_i)\]
    since $\delta$ is a permutation.

    Lastly, easy to see $\alpha$ is bijection or one can write out the inverse map of $\alpha$. Hence $\CoInd_{U_l}^G T$ is isomorphice to $\F_p[G/U_l]\otimes_{\F_p} T$ as $\F_p[G]-$modules.
\end{proof}

Here, we computed $H^r(U_l,T)$. Next, we will compute $H^r_{\Iw}(N,T)$. We need the following propositions due to Tate \cite{MR0429837}. 

\begin{proposition}\label{Tate}
 Suppose $i>0$ and $M=\varprojlim M_l$ where each $M_l$ is a finite discrete $G$-module. If $H^{i-1}(G,M_l)$ is finite for every $l$, then
 \[
 H^i(G,M)=\varprojlim_l H^i(G,M_l)
 \]
\end{proposition}

\begin{lemma}\label{shapiro}
      Let $G/N\cong\Z_p$ and $U_l$ be the unique open normal subgroup of $G$ containing $N$ such that $G/U_l\cong \Z/p^l\Z$. And $T$ is a finite $\F_p[G]-$module.  Then  $H^r_{\Iw}(N,T)\cong H^r(G,\F_p[[G/N]]\otimes_{\F_p} T)$.
\end{lemma}
\begin{proof}
    By lemma \ref{action1}, 
\begin{equation*}
\begin{split}
          H^r_{\Iw}(N,T)&= \varprojlim_{N\leq U_l \trianglelefteq^{\circ} G  } H^r(U_l,T)\\
          &=\varprojlim_{N\leq U_l \trianglelefteq^{\circ} G  }H^r(G,\F_p[G/U_l]\otimes_{\F_p} T)\\
          &= H^r(G,\varprojlim_{N\leq U_l \trianglelefteq^{\circ} G  }\F_p[G/U_l]\otimes_{\F_p} T)\\
          &= H^r(G,\F_p[[G/N]]\otimes_{\F_p} T).
\end{split}
\end{equation*} 

To prove the third equality is true, we need to check that it satisfies conditions in Proposition \ref{Tate}. First, we have that  $\F_p[G/U_l]\otimes_{\F_p} T$ is finite module. Second, by \cite{MR4537772}*{Proposition 2.2.2}, we know that $H^i(U_l,T)\cong H^i(G,\F_p[G/U_l]\otimes_{\F_p} T)$ is finitely generated $\F_p[G/U_l]$ module for all $i\geq 0$. Since $\F_p[G/U_l]$ is finite, we have that $H^i(G,\F_p[G/U_l]\otimes_{\F_p} T)$ is also finite. 

For the last equality, note that $T$ is a finite dimensional vector space over $F_p$. Hence $ \varprojlim \F_p[G/U_l]\otimes_{\F_p} T\cong \F_p[[G/N]]\otimes_{\F_p} T $ as $F_p$ module. One can see that the isomorphism also preserves $G$ action. Hence it is also a $\F_p[G]$- isomorphism. 

\end{proof}

 \begin{remark}
     We know $\varprojlim\CoInd_{U_l}^G T\coloneqq \varprojlim\Hom_{\Z U_l}(\Z G, T)\cong \F_p[[G/N]]\otimes_{\F_p} T $. But $ \CoInd_{N}^G T\coloneqq \Hom_{\Z N}(\Z G, T)$ may not be isomorphic to $\varprojlim\CoInd_{U_l}^G T$. That is one reason why we separate the proof of lemma \ref{action1} and lemma  \ref{shapiro} and we often write our induced module as $\Omega\otimes T$ instead of $\Hom_{\Z U_l}(\Z G, T).$
 \end{remark}

Let $\G$ be a group containing $G$ and $N$ as normal subgroups such that $\G/G=\Delta$ is an abelian group and $\G/N$ is also abelian and $\G/N\cong \Delta \oplus G/N$, where $G/N$ is a finitely generated pro-$p$ quotient. Let $T$ be a $\F_p[\G]$-module and $\Omega=\F_p[[G/N]]$.  Next, We will define an action of group $\Delta$ on $H^r(G,\Omega\otimes T)$, $H^r(G,I^n/I^{n+1}\otimes T)$, $H^r(G,\Omega/I^n\otimes T)$, and $ H^r(G,T)$. And we will show that the generalized Bockstein map $\Psi^{(n)}$ preserves the action. We take $G/N\cong \Z_p$ or $\Z/p^l\Z$ in our paper. But the setting up works generally. 

Define $\tau\in \G$ acting on $\sum_{i}\bar{\sigma}_i\otimes t_i\in \Omega\otimes T$ as $ \tau(\sum_{i}\bar{\sigma}_i\otimes t_i)=\sum_{i}\bar{\sigma}_i\otimes \tau t_i$. Recall the action $g\in G$ on $\Omega\otimes T$ is $ g(\sum_{i}\bar{\sigma}_i\otimes t_i)=\sum_{i}\bar{g}\bar{\sigma}_i\otimes gt_i$. These two actions have different effects when $\tau \in G\subset \G$. But they have the same effect when $\tau\in N\subset \G$.
For every $\tau\in \G$, we have a group homomorphism  $\alpha: G\rightarrow G, g\rightarrow \tau^{-1}g\tau$ and a module homomorphsim $\beta: \Omega\otimes T\rightarrow \Omega\otimes T, \sum_{i}\bar{\sigma}_i\otimes t_i\rightarrow \tau(\sum_{i}\bar{\sigma}_i\otimes t_i)$. And
\begin{equation*}
    \begin{split}
        \beta(\alpha(g) (\sum_{i}\bar{\sigma}_i\otimes t_i))&=\beta(\sum_{i}(\overline{\tau^{-1}g\tau})\bar{\sigma}_i\otimes \tau^{-1}g\tau t_i)\\
        &=\sum_{i}(\overline{\tau^{-1}g\tau})\bar{\sigma}_i\otimes \tau \tau^{-1}g\tau t_i\\
        &=\sum_{i}\bar{g}\bar{\sigma}_i\otimes g\tau t_i\\
        &=g(\beta(\sum_{i}\bar{\sigma}_i\otimes t_i))
    \end{split}
\end{equation*}

 By the functorial properties of the cohomology groups \cite{milneCFT}*{Chapter 2,p.~66}, we have a homomorphism $H^r(G,\Omega\otimes T)\rightarrow H^r(G,\Omega\otimes T)$. This gives us an action of $\G$ on $ H^r(G,\Omega\otimes T)$.

 We know the action $\tau\in \G$ on $\Omega\otimes T$ and the action $G$ on $\Omega\otimes T$ have the same effect when $\tau\in N$. By a well known fact \cite{milneCFT}*{Example 1.27(d),p.~67}, the induced action  $\tau\in N\subset \G$ on $H^r(G,\Omega\otimes T)$ is trivial. So the action $\G$ on $H^r(G,\Omega\otimes T)$ factors through $\G/N\cong \Delta\oplus G/N$. Hence, we get an action $\Delta$ on $H^r(G,\Omega\otimes T)$ by viewing $\Delta$ as a subgroup of $\G/N$.

Similarly, we can define the action of $\G$ on $\Omega/I^n\otimes T$ and $I^n/I^{n+1}\otimes T$ in the same way. And actions are compatible, i.e.
\[
\begin{tikzcd}
0\arrow[r] 
& I^n/I^{n+1}\otimes T \arrow[r] \arrow[d,"\tau"]&\Omega/I^{n+1}\otimes T\arrow[r] \arrow[d,"\tau"] &\Omega/I^n\otimes T\arrow[r] \arrow[d,"\tau"] &0 \\
0\arrow[r] 
& I^n/I^{n+1}\otimes T \arrow[r] &\Omega/I^{n+1}\otimes T\arrow[r] &\Omega/I^n\otimes T\arrow[r]  &0
\end{tikzcd}\]

Let $\cdots\rightarrow\Z G^{\oplus k}\rightarrow \Z G^{\oplus k-1}\rightarrow\cdots \rightarrow\Z G^{\oplus 2}\rightarrow \Z G\rightarrow \Z \rightarrow 0$ be a $\Z[G]$ projective resolution of $\Z$. Then $(\alpha,\beta)$ defines a homomorphism of complexes
\[
\Hom(\Z G^{\oplus k},\Omega\otimes T)\rightarrow \Hom(\Z G^{\oplus k},\Omega\otimes T), \phi\rightarrow \beta\circ \phi\circ \alpha^k
\]
for any $k$. And it induces a homomorphism between two exact sequences of complexes:
\[
\begin{tikzcd}[sep=tiny,font=\scriptsize
]
0\arrow[r] 
& \Hom(\Z G^{\oplus k},I^n/I^{n+1}\otimes T)\arrow[r] \arrow[d]&\Hom(\Z G^{\oplus k},\Omega/I^{n+1}\otimes T)\arrow[r] \arrow[d] &\Hom(\Z G^{\oplus k},\Omega/I^n\otimes T)\arrow[r] \arrow[d] &0 \\
0\arrow[r] 
& \Hom(\Z G^{\oplus k},I^n/I^{n+1}\otimes T) \arrow[r] &\Hom(\Z G^{\oplus k},\Omega/I^{n+1}\otimes T)\arrow[r] &\Hom(\Z G^{\oplus k},\Omega/I^n\otimes T)\arrow[r]  &0
\end{tikzcd}
\]
By the functorial property of cohomology, we get a homomorphism between two long exact sequences.
\[
\begin{tikzcd}[sep=tiny,font=
\footnotesize
]
\cdots\arrow[r] &H^r(G,\Omega/I^{n+1}\otimes T)\arrow[r] \arrow[d] &H^r(G,\Omega/I^n\otimes T)\arrow[r] \arrow[d] & H^{r+1}(G,I^n/I^{n+1}\otimes T)\arrow[d] \arrow[r] & \cdots\\
\cdots\arrow[r] &H^r(G,\Omega/I^{n+1}\otimes T)\arrow[r] &H^r(G,\Omega/I^n\otimes T)\arrow[r]  & H^{r+1}(G,I^n/I^{n+1}\otimes T) \arrow[r] &\cdots
\end{tikzcd}
 \]
Each column gives an action of $\tau\in \G$ and they are compatible with each other. In particular, since the generalized Bockstein map is the connecting mapping, we have $\Psi^{(n)}(\tau \phi)=\tau \Psi^{(n)}(\phi)$ for any $\phi \in H^{d-1}(G,T\otimes \Omega/I^n)$. As before, all actions factor through $\G/N$. It induces an action of $\Delta$ on the cohomology group. And the generalized Bockstein map $\Psi^{(n)}$ preserves the actions. 
\begin{remark}
    By lemma \ref{action1}, we have $\Hom_{\Z U_l}(\Z G, T)\cong \F_p[G/U_l]\otimes_{\F_p} T$. Corresponding to the action $\tau\in \G$ on $ \F_p[G/U_l]\otimes_{\F_p} T$, the action  $\tau\in \G$ on $\phi\in \Hom_{\Z U_l}(\Z G, T)$ is $(\tau\phi)(g)=\tau\phi(\tau^{-1}g\tau)$.
\end{remark}

 \begin{remark}\label{skip1}
     The action $\Delta$ on the cohomology group is a left action. 
 \end{remark}

 Now we introduce a new notation to express elements in $\Omega$ which will simplify our future calculation. For simplicity, assume  $G/N\cong \Z_p $ or $\Z/p^l\Z$ and $\sigma$ is the generator of $ G/N$. Let $x=\sigma-1$, then $I=<\sigma-1>=<x>$ and $\Omega=\F_p[[G/N]]\cong \F_p[[x]]$ or $\F_p[x]/(x^{p^l})$ with respect to $G/N\cong \Z_p $ or $\Z/p^l\Z$. Elements in $\Omega\otimes T$ can be write as 
 \[\sum_i \sigma^i\otimes t_i=\sum_i(1+x)^i\otimes t_i=\sum_i x^i\otimes \psi_i\] 
 for some $\psi_i\in T$. Let $\chi$ be the group homomorphism $\chi:G\rightarrow G/N\cong \Z_p$ or $\Z/p^l\Z$. Then the action $g\in G$ on $ \Omega\otimes T$ can be writen as
 \[g(\sum_i x^i\otimes \psi_i)=\sum_i (1+x)^{\chi(g)}x^i\otimes g\psi_i=\sum_i (x^i\otimes (\sum_{k=0}^{k=i}\binom{\chi(g)}{k}g \psi_{i-k}))\]
\subsection{Massey product} \label{Massey product sect}
In this section, we recall some facts of Massey products, defining systems\cite{MR4305541}\cite{MR3584563} and proper defining systems\cite{MR4537772}. More reference are \cite{MR202136}\cite{MR385851}

For the classical definition of Massey products in group cohomology. Let $A$ be a commutative ring with trivial $G$ action and discrete topology. By \cite{MR2392026}, we know that inhomogeneous  continuous cochains $\mathcal{C}^.=\bigoplus_{k\geq0}\mathcal{C}^k(G,A) $ form a differential graded algebra over $A$.
It equips with product $\cup$ and differential $d:\mathcal{C}^n\rightarrow \mathcal{C}^{n+1}$ such that: $d(a\cup b)=(d\, a)\cup b +(-1)^ka\cup (d\, b)$ where $a\in \mathcal{C}^k$ and $d^2=0$. We have $H^n(G,A)=\ker\,d^n/\Ima\,d^{n-1}$. 

Let $U_{n+1}(A)$ be the group of $(n+1)\times (n+1)$ upper triangular matrix with diagonal entries equal 1. Let $Z(U_{n+1}(A))$ be the center of $U_{n+1}(A)$. It is a group of matrix whose diagonal entries equal to 1 and other only possible non zero entry is in the spot $(1,n+1)$. Let $\bar{U}_{n+1}(A)=U_{n+1}(A)/Z(U_{n+1}(A))$. 

Let $\chi_1,\chi_2,\cdots,\chi_n$ be $n$ elements in $H^1(G,A)=\Hom(G,A)$. The defining system with respect to $ \chi_1,\chi_2,\cdots,\chi_n$ is a homomorphism $\bar{\rho}: G\rightarrow \bar{U}_{n+1}(A)$ with $\rho_{i,i+1}=\chi_i$, where $\rho_{i,j}$ is the composition of map $\bar{\rho}:G\rightarrow \bar{U}_{n+1}(A)$ with the projection of $\bar{U}_{n+1}(A)$ to its $(i,j)$ entries.  One can check: $\sum_{j=2}^{j=n}\rho_{1,j}(g_1)\rho_{j,n+1}(g_2)$ is a cocyle inside $ \mathcal{C}^2$ which represents an element in $H^2(G,A)$. We denote the element as $ (\chi_1,\chi_2,\cdots,\chi_n)_{\bar{\rho}}$ and call it the Massey product with respect to the defining system $\bar{\rho}$. By \cite{MR385851}, the Massey product $ (\chi_1,\chi_2,\cdots,\chi_n)_{\bar{\rho}}$ vanishing is equivalent to that $\Bar{\rho}$ can be lifted to a homomorphism $\rho: G\rightarrow U_{n+1}(A)$.

Here $\bar{\rho} \in \Hom(G, \bar{U}_{n+1}(A))$ can be think as degree one cocycle in  $\mathcal{C}^1(G,\bar{U}_{n+1}(A))$ with $G$ acting trivially on $\bar{U}_{n+1}(A)$. In general, the action of $G$ on the ring $A$ may not be trivial. We could define defining systems as a degree one cocycle in the same philosophy and it can be applied in general situations. References are \cite{MR4537772}. We directly borrow definitions from Section 3 of \cite{MR4537772} without giving definitions again here.

But our definition of proper defining systems is a little different from the definition in \cite{MR4537772}. We give a new definition here.

\begin{definition}
    Let $\chi\in H^1(G,T_1)$ and $\psi_0\in H^1(G,T_2)$. Then we call the defining system $\bar{\rho}:G\rightarrow \mathcal{U}(\mathcal{A})$ with respect to $\underbrace{\chi,\chi,\cdots.\chi}_\text{n copies},\psi_0$ as proper defining system if $\bar{\rho}$ is of the following forms:
    \[
    \begin{bmatrix}  
    1 & \chi &\binom{\chi}{2}&\binom{\chi}{3}&\binom{\chi}{4}&\cdots & *\\ 
    0 &1& \chi &\binom{\chi}{2}&\binom{\chi}{3}&\cdots & \psi_{n-1}\\
    \vdots& \vdots& \vdots& \vdots& \vdots& \ddots& \vdots\\
   0&0&0& 1&\chi&\binom{\chi}{2}&\psi_2\\
   0&0&0&0&1&\chi&\psi_1\\
   0&0&0&0&0&1&\psi_0\\
   0&0&0&0&0&0&1
    \end{bmatrix}
    \]
\end{definition}
\begin{remark}
    In \cite{MR4537772}, when they define the proper defining system, they first divide the matrix into four blocks that looks like the one in the following Lemma \ref{blocklemma}. And then they fix the up-left block and down-right block and let the up-right block varies. Here, our definition of the proper defining system can be viewed as a special case of them. In our definition, we fix and give a explicit form of the first $(n+1)\times (n+1)$ block and we let the last column varies except $1$ and $\psi_0$.
\end{remark}

The following Lemma \ref{blocklemma} and Remark \ref{remark operation} will only be used in the section of numerical criterion. One can skip if not interested in the numerical criterion. 
\begin{lemma}\label{blocklemma}
    Let $\Bar{\rho}: G\rightarrow \Bar{U}_{m+n}$ be a defining system. We can write the homomorphism $\Bar{\rho}$ as a block matrix:
    \[
    \Bar{\rho}=\begin{bmatrix}
        A_n&\Bar{B}_{n,m}\\
        0& D_m
    \end{bmatrix}
    \]
    where $A_n$ is a $n\times n$ matrix and $D_m $ is a $m\times m$ matrix. $\Bar{B}_{n,m}$ is a $n\times m$ matrix without $(n,m)$-entry. Let $ \Bar{\rho}': G\rightarrow \Bar{U}_{m+n}$ be another defining system with the same first $n$  columns and last $m$ rows as $ \Bar{\rho}$, i.e. 
    \[
    \Bar{\rho}'=\begin{bmatrix}
        A_n&\Bar{B}_{n,m}'\\
        0& D_m
    \end{bmatrix}
    \]

    Then \[
    \begin{bmatrix}
        A_n&\Bar{B}_{n,m}+\Bar{B}_{n,m}'\\
        0& D_m
        \end{bmatrix}
    \]
    is also a defining system.
\end{lemma}

\begin{proof}
    The proof is a trivial calculation of matrices by the definition of defining system. We omit here.
\end{proof}
\begin{remark}\label{remark operation}
    This easy observation gives us a way to generate a new defining system from old defining systems. For proper defining system, let $\Bar{\rho}_n: G\rightarrow \Bar{U}_{n+1} $ be a proper defining system, i.e,
    \[
    \Bar{\rho}_n=\begin{bmatrix}  
    1 & \chi &\binom{\chi}{2}&\binom{\chi}{3}&\binom{\chi}{4}&\cdots & *\\ 
    0 &1& \chi &\binom{\chi}{2}&\binom{\chi}{3}&\cdots & \psi_{n-2}\\
    \vdots& \vdots& \vdots& \vdots& \vdots& \ddots& \vdots\\
   0&0&0& 1&\chi&\binom{\chi}{2}&\psi_2\\
   0&0&0&0&1&\chi&\psi_1\\
   0&0&0&0&0&1&\psi_0\\
   0&0&0&0&0&0&1
    \end{bmatrix}
    \]
One can check that $\Bar{\rho}_{n+m}:G\rightarrow \Bar{U}_{m+n+1}$ induced by $ \Bar{\rho}_n$ is still a proper defining system:
\[\Bar{\rho}_{n+m}=\begin{bmatrix}  
    1 & \chi &\binom{\chi}{2}&\binom{\chi}{3}& 
    \binom{\chi}{4}& \cdots &\cdots & *\\ 
    0 &1& \chi &\binom{\chi}{2}&\binom{\chi}{3}& \cdots
    &\cdots & \psi_{n-2}\\
    \ddots& \ddots& \ddots& \ddots& \ddots& \vdots &\vdots& \vdots\\
    \cdots & 0& 1& \chi& \binom{\chi}{2}& \ldots& \binom{\chi}{m}& \psi_0\\
    0& \cdots& 0 & 1& \chi& \cdots& \binom{\chi}{m-1}& 0\\
    \vdots& \vdots& \vdots& \vdots& \vdots& \vdots &\vdots& \vdots\\
   0&0&0&0& 1&\chi&\binom{\chi}{2}&0\\
   0&0&0&0&0&1&\chi&0\\
   0&0&0&0&0&0&1&0\\
   0&0&0&0&0&0&0&1
    \end{bmatrix}
\]

If we have another proper defining system $ \Bar{\rho}_{n+m}':G\rightarrow \Bar{U}_{n+m+1}$, 
\[
\Bar{\rho}_{n+m}'=\begin{bmatrix}  
    1 & \chi &\binom{\chi}{2}&\binom{\chi}{3}& 
    \binom{\chi}{4}& \cdots &\cdots & *\\ 
    0 &1& \chi &\binom{\chi}{2}&\binom{\chi}{3}& \cdots
    &\cdots & \psi_{n+m-2}'\\
    \ddots& \ddots& \ddots& \ddots& \ddots& \vdots &\vdots& \vdots\\
    \cdots & 0& 1& \chi& \binom{\chi}{2}& \ldots& \binom{\chi}{m}& \psi_{m}'\\
    0& \cdots& 0 & 1& \chi& \cdots& \binom{\chi}{m-1}& \psi_{m-1}'\\
    \vdots& \vdots& \vdots& \vdots& \vdots& \vdots &\vdots& \vdots\\
   0&0&0&0& 1&\chi&\binom{\chi}{2}&\psi_2'\\
   0&0&0&0&0&1&\chi&\psi_1'\\
   0&0&0&0&0&0&1&\psi_0'\\
   0&0&0&0&0&0&0&1
    \end{bmatrix}
\]
Then by Lemma \ref{blocklemma}, we can produce a new proper defining system

\[
\begin{bmatrix}  
    1 & \chi &\binom{\chi}{2}&\binom{\chi}{3}& 
    \binom{\chi}{4}& \cdots &\cdots & *\\ 
    0 &1& \chi &\binom{\chi}{2}&\binom{\chi}{3}& \cdots
    &\cdots & \psi_{n+m-2}'+\psi_{n-2}\\
    \ddots& \ddots& \ddots& \ddots& \ddots& \vdots &\vdots& \vdots\\
    \cdots & 0& 1& \chi& \binom{\chi}{2}& \ldots& \binom{\chi}{m}& \psi_{m}'+\psi_0\\
    0& \cdots& 0 & 1& \chi& \cdots& \binom{\chi}{m-1}& \psi_{m-1}'\\
    \vdots& \vdots& \vdots& \vdots& \vdots& \vdots &\vdots& \vdots\\
   0&0&0&0& 1&\chi&\binom{\chi}{2}&\psi_2'\\
   0&0&0&0&0&1&\chi&\psi_1'\\
   0&0&0&0&0&0&1&\psi_0'\\
   0&0&0&0&0&0&0&1
    \end{bmatrix}
\]
\end{remark}

\vskip3mm

The image of generalized Bockstein map $\Ima\Psi^{(n)}$ is spanned by certain Massey products relative to a proper defining system. The next theorem is just a special case of theorem 4.3.1 in \cite{MR4537772}.
\begin{theorem}[ LLSWW\cite{MR4537772}] \label{LLSWW}
    Let $G$ be a group with $p$-cohomological dimension $2$. View $ \Omega/I^n\otimes T$ as a quotient of polynomial ring in terms of variable $x$ with coefficient in $T$, Let $f(\sigma)=\psi_0(\sigma)+\psi_1(\sigma)x+\cdots +\psi_{n-1}(\sigma)x^{n-1}$ be a cocycle in $\mathcal{C}^1(G, \Omega/I^n \otimes T)$, where $\psi_i$ is a cochain in $\mathcal{C}^1(G,T)$. Then $\Psi^{(n)}(f)=(\sum_{i=1}^{n}\binom{\chi}{i}\cup\psi_{n-i})x^n$, where $\chi$ is the quotient map $\chi: G\rightarrow G/N\cong \Z_p $ or $\Z/p^l\Z$. And easy to see, $\sum_{i=1}^{n}\binom{\chi}{i}\cup\psi_{n-i}$ is Massey product $(\chi^{(n)},\psi_0)$ relative to the proper defining system. Here $\chi^{(n)}$ denotes $n$- copies of $\chi$. More precisely, the proper defining system is:
    \[ \begin{bmatrix}  
    1 & \chi &\binom{\chi}{2}&\binom{\chi}{3}&\binom{\chi}{4}&\cdots & *\\ 
    0 &1& \chi &\binom{\chi}{2}&\binom{\chi}{3}&\cdots & \psi_{n-1}\\
    \vdots& \vdots& \vdots& \vdots& \vdots& \ddots& \vdots\\
   0&0&0& 1&\chi&\binom{\chi}{2}&\psi_2\\
   0&0&0&0&1&\chi&\psi_1\\
   0&0&0&0&0&1&\psi_0\\
   0&0&0&0&0&0&1
    \end{bmatrix}\]
\end{theorem}
\begin{remark}
    We have $ (\binom{\chi}{i}\cup\psi_{n-i})(g_1,g_2)=\binom{\chi(g_1)}{i}\cup g_1\psi_{n-i}(g_2)$. And this is compatible with our action of $G$ on $\Omega\otimes T$. And this is compatible with definition of defining system $\rho$ to be a cocycle: $\rho(g_1g_2)=\rho(g_1)g_1\rho(g_2)$
\end{remark}

\section{Size of $H^2$} \label{Size of $H^2$}

The philosophy of the strategy is to use Massey products to analyze the size of $H^2$. We first prove some lemmas we will use later. Since they all can be derived purely from group cohomology theory. We put them in this section. From now, we assume $G$ has $p$ cohomological dimension $d=2$.
\begin{lemma}\label{main lemma}
Let $G$ be a profinite group with $p$ cohomological dimension equal 2, $N$ be a closed normal subgroup such that $G/N\cong \Z_p$ or $\Z/p^l\Z$. And $T$ is a $\F_p[G]$ module. Assume $H^2(G,T)\cong \F_p$ and $\Psi^{(k)}\neq 0$ for some $ 0<k< \#G/N$, then $\#H_{\Iw}^2(N,T)=p^n$ where $n=\min\{n|\Psi^{(n)}\neq 0\}$  
\end{lemma}
\begin{proof}
Take $ n=\min\{n\in \N|\Psi^{(n)}\neq 0\}$, then $\Psi^{(n)}\neq 0 $.  We have $H^2(G,T)\cong \F_p\cong \Ima\Psi^{(n)}$. By theorem \ref{filtration}, $I^n H_{\Iw}^2(N,T)=I^{n+1}H_{\Iw}^2(N,T)$. By Nakayama's lemma, $ I^n H_{\Iw}^2(N,T)$ $=0$. We have the filtration $ H_{\Iw}^2(N,T)\supset I H_{\Iw}^2(N,T)\supset I^2H_{\Iw}^2(N,T)\supset \cdots \supset I^n H_{\Iw}^2(N,T)=0$ and when $i<n$, $ \frac{I^i H_{\Iw}^2(N,T)}{I^{i+1}H_{\Iw}^2(N,T)}=\F_p$. Hence $\#H^2(N,T)=p^n $.
\end{proof}
\begin{theorem}\label{reduce}
    Let $G$ be a profinite group, $N$ be a closed normal subgroup such that $G/N\cong \Z_p$ or $\Z/p^l\Z$ and $T$ be a $\F_p[G]$ module. Fix an integer $n<\#G/N$. Assume the generalized Bockstein map $\Psi^{(i)}=0$ for all $0\leq i<n$, then the value $\Psi^{(n)}(f)$,  where $f(\sigma)=\psi_0(\sigma)+\psi_1(\sigma)x+\cdots +\psi_{n-1}(\sigma)x^{n-1}$ is a cocycle in $\mathcal{C}^1(G, \Omega/I^n\otimes T)$, only depends on the cohomology class of $\psi_0$ and does not depend on other coefficients $\psi_i$, $0<i<n$. 
\end{theorem}
\begin{proof}
    Let $ f'=\psi_0'+\psi_1'x+\cdots +\psi_{n-1}'x^{n-1}$ be another cocycle in $\mathcal{C}^1(G, \Omega/I^n \otimes T)$ such that $\psi_0$ and $\psi_0'$ are in the same cohomology class. Then $(\psi_0-\psi_0')(\sigma)=\sigma m-m $ for some $m\in T$. Let $\delta= (\psi_0-\psi_0')+(\chi\cup m)x+ (\binom{\chi}{2}\cup m)x^2+\cdots + (\binom{\chi}{n-1}\cup m)x^{n-1}$. One can check that $\Psi^{(n)}(\delta)=d(\binom{\chi}{n}\cup m) x^n =0 \in H^2(G,T)\otimes I^n/I^{n+1}$. By adding $\delta$ to $f'$, we can assume that $\psi_0'=\psi_0$. Then $f-f'\in \mathcal{C}^1(G,I/I^n\otimes T )$. Since $\Omega/I^{n-1}\cong \F_p[x]/(x^{n-1})$ is isomorphic to $I/I^n\cong (x)/(x^{n})$ by a multiplication of $x$, $\Psi^{(n)}(f-f')=\Psi^{(n-1)}(\frac{f-f'}{x})=0$ by assumption. Hence $\Psi^{(n)}(f)$ only depends on the cohomology class of $\psi_0$.   
\end{proof}
Under the conditions  $\Psi^{(i)}=0$ for all $0\leq i<n$, for convenience, we use $\Psi^{(n)}([\psi_0]) $ to denote $\Psi^{(n)}(\psi_0+\psi_1x+\cdots +\psi_{n-1}x^{n-1})$ since the value only depends on the cohomology class of $\psi_0$. And in the language of Massey product, $\Psi^{(n)}([\psi_0]) $ equals to the Massey product $(\chi^{(n)},\psi_0)$ relative to a proper defining system.

Notice from the following exact sequence:
\[
H^1(G,\Omega/I^{n+1}\otimes T)\rightarrow H^1(G,\Omega/I^n\otimes T)\rightarrow H^2(G, I^n/I^{n+1}\otimes T)
\]
Hence $\Psi^{(n)}(f)$ is zero for a cocycle $f(\sigma)=\psi_0(\sigma)+\psi_1(\sigma)x+\cdots +\psi_{n-1}(\sigma)x^{n-1}$ in $\mathcal{C}^1(G, \Omega/I^n\otimes T)$  if and only if we can lift $f$ to $\mathcal{C}^1(G, \Omega/I^{n+1}\otimes T) $, i.e. there is a cocycle in $\mathcal{C}^1(G,\Omega/I^{n+1}\otimes T)$ in the form of  $\Tilde{f}(\sigma)=\psi_0(\sigma)+\psi_1(\sigma)x+\cdots +\psi_{n-1}(\sigma)x^{n-1}+\psi_n(\sigma)x^n$. 
\begin{definition}
    Let $ \psi$ be an element in $ H^1(G,T)=H^1(G,\Omega/I\otimes T)$. We say $\psi$ has $p$ cyclic Massey product vanishing property, if we can lift $\psi$ to an element in $H^1(G, \Omega\otimes T)$, i.e. there is an element in $\mathcal{C}^1(G,\Omega\otimes T)$ in the form $\sum_i \psi_i(\sigma)x^i$ with $\psi_0=\psi$, where the sum over all $0\leq i< \infty$ or $0\leq i< p^k$ with respect to $G/N\cong \Z_p$ or $\Z/p^k\Z$. 
\end{definition}

\begin{remark}
    Assume  $\Psi^{(i)}=0$ for all $0\leq i<n$ and $\psi$ has $p$ cyclic Massey product vanishing property. Then we have $ \Psi^{(n)}([\psi])=0$.
\end{remark}

\begin{theorem}\label{cyclicvanish}
   The element  $\psi\in H^1(G,T)$ has $p$ cyclic Massey product vanishing property if and olny if $\psi\in \Ima(H_{\Iw}^1(N,T)\xrightarrow{Cor}H^1(G,T))$ .
\end{theorem}

\begin{proof}
By lemma \ref{shapiro}, we have $ H_{\Iw}^1(N,T)\cong H^1(G,\Omega\otimes T)$. And the map $H^1(G,\Omega\otimes T)\rightarrow H^1(G,\Omega/I\otimes T) $ induced by the quotient $\Omega\rightarrow\Omega/I$ corresponding to the correstriction map $  H_{\Iw}^1(N,T)\rightarrow H^1(G,T)$ by definition.
  
\end{proof}

\section{Applications to number theory} \label{Applications to number theory}

Now we apply previous sections to number theory. Though all notations are standard, we list them here for convenience:
\begin{enumerate}
\item $p$: odd prime.
   \item $K$: number field.
   \item $ \mu_{p^l}$: group of $p^l$-th roots of unity.
  \item $S$: set of primes of $K$ above $p$. 
  \item $K_S$: maximal algebraic extension of $K$ which is unramified outside primes above $p$ and infinite primes.
   \item $G_{K,S}=\Gal(K_S/K)$.
  \item $\OO_K$: ring of integers of $K$.
 \item $\OO_{K,S}$: ring of $S$ -integers of $K$.
    \item $\Cl(K)$: class group of $K$.
       \item $h_K$: the size of $\Cl(K)$.
    \item $\Cl_S(K)$: $S$ -class group of $K$.

   \item $Br(\OO_K[1/p])$: The subgroup of the Brauer group $Br(K)$ of $K$ consisting of all classes of central simple $K$-algebras, which are split outside of primes above $p$. Hence, we have $Br(\OO_K[1/p])\cong (\Q/\Z)^{\#S-1} $ as abelian group. 

   \item $ A[n]$: subgroup of elements of abelian group $A$ that annihilated by $n$.
   \item Let $K\subset K_1\subset K_2\subset \cdots \subset K_\infty$ be a $\Z_p$ extension of $K$.
   \item $X$: the Iwasawa module $X=\varprojlim \Cl(K_l)$.
   \item $\lambda$: the Iwasawa invariant $ \lambda$ for $X$.
   \item $ \chi$: character that equal the restriction  $ G_{K,S}\xrightarrow{Res} \Gal(K_\infty/K)\cong \Z_p$.

\end{enumerate}

\begin{theorem}\label{cv in NT}
    Let $K_l/K$ be a $\Z/p^l\Z$ extension of $K$ inside $K_S$. Take $G=G_{K,S}=\Gal(K_S/K)$, $N=G_{K_l,S}=\Gal(K_S/K_l)$, $T=\mu_p$. Let  $\alpha \in \Nrm_{K_l/K}(\OO_{K_l,S})$, then $\alpha$ has $p$ cyclic Massey product vanishing property.

Let $K_\infty/K$ be a $\Z_p$ extension of $K$ inside $K_S$. Take $G=G_{K,S}=\Gal(K_S/K)$, $N=G_{K_\infty,S}=\Gal(K_S/K_\infty)$, $T=\mu_p$. Let $ \alpha \in \Ima(\varprojlim\OO_{K_l,S}\xrightarrow{\Nrm} K)$ where the inverse limit is taking over all sub-extension of $K_\infty/K$ with respect to norm map. Then $\alpha$ has $p$ cyclic Massey product vanishing property.
\end{theorem}
\begin{proof}
    By Kummer theory, $H^1(G,\mu_p)\cong K^* \cap (K_S^*)^p/(K^*)^p$ and $H^1(N,\mu_p)=K_l^* \cap (K_S^*)^p/(K_l^*)^p$. The corestriction map from $H^1(N,\mu_p)$ to $ H^1(G,\mu_p)$  corresponds to Norm map\cite{milneCFT}. Notice that $\OO_{K_l,S} \in (K_S^*)^p$.
    And by theorem \ref{cyclicvanish}, the theorem holds. The second conclusion follows by taking the inverse limit.
\end{proof}
\begin{remark}
    The map $\varprojlim\OO_{K_l,S}\xrightarrow{\Nrm} \OO_{K,S}\hookrightarrow K$ is the projection map from the inverse limit $\varprojlim\OO_{K_l,S} $ to the first factor. 
\end{remark}

From Kummer theory, we have the following two exact sequences: (good references are \cite{kolster}\cite{MR4537772}\cite{MR2392026})
\begin{equation}
    0\rightarrow\OO_{K,S}^*/p:=\OO_{K,S}^*/(\OO_{K,S}^*)^p\rightarrow H^1(G_{K,S},\mu_p)\rightarrow \Cl_S(K)[p]\rightarrow 0 
  \label{1}  
\end{equation}

\begin{equation}
   0\rightarrow \Cl_S(K)/p  \rightarrow H^2(G_{K,S},\mu_p)\rightarrow Br(\OO_K[1/p])[p]\rightarrow 0 
  \label{2} 
\end{equation}

The map $\OO_{K,S}^*/p\rightarrow H^1(G_{K,S},\mu_p) $ is the composite of the natural inclusion and Kummer map,i.e.: $ \OO_{K,S}^*/p\rightarrow K^* \cap (K_S^*)^p/(K^*)^p \cong H^1(G_{K,S},\mu_p)$. The map $ H^1(G_{K,S},\mu_p)\rightarrow \Cl_S(K)[p]$ is the map $ \alpha \in K^* \cap (K_S)^p/(K^*)^p \cong H^1(G_{K,S},\mu_p)\rightarrow [I] \in \Cl_S(K)[p]$ where $ I$ is an ideal such that $ I^p=\alpha \OO_{K,S}$. 
Since $ Br(\OO_K[1/p])\cong (\Q/\Z)^{\#S-1}$, so $Br(\OO_K[1/p])[p]\cong \F_p^{\#S-1}$

Now, we can state and prove our main theorem. 

\begin{theorem} \label{main theom}
    Let $K\subset K_1\subset K_2\subset \cdots \subset K_\infty$ be a $\Z_p$ extension of $K$ and $S$ be the set of primes above $p$ for $K$. Assume all primes in $S$ are totally ramified in $K_\infty/K$. Let $X_{cs}=\varprojlim\Cl_S(K_l)$ and $\mu_{cs}$, $\lambda_{cs}$  be the Iwasawa invariant of $X_{cs} $. Assume $X_{cs} $ has no torsion element and $H^2(G_{K,S},\mu_p)\cong \F_p$. 
    
    Then $\mu_{cs}=0$ if and only if there exists $k$ such that $\Psi^{(k)}\neq 0$ for some $k$.  If $\mu_{cs}=0$, then $\lambda_{cs}=\min\{n|\Psi^{(n)}\neq 0\}-\#S+1$.
\end{theorem}
\begin{proof}
We have the following exact sequence:
    \begin{equation*}
   0\rightarrow \Cl_S(K_l)/p  \rightarrow H^2(G_{K_l,S},\mu_p)\rightarrow Br(\OO_{K_l}[1/p])[p]\rightarrow 0 
\end{equation*}
for every $l$.  since $\Cl_S(K_l)/p$ is finite group, it satisfies Mittag-Leffler condition. Thus the above exact sequence remains exact after taking inverse limit.
\begin{equation*}
   0\rightarrow \varprojlim \Cl_S(K_l)/p  \rightarrow \varprojlim H^2(G_{K_l,S},\mu_p)\rightarrow \varprojlim Br(\OO_{K_l}[1/p])[p]\rightarrow 0 
\end{equation*}

Since  all primes in $S$ are totally ramified $K_\infty/K$, we have $Br(\OO_{K_l}[1/p])[p]\cong Br(\OO_{K_{l+1}}[1/p])[p]:[A]\rightarrow [A\otimes_{K_l} K_{l+1}]$ for every $l$, where $[A]$ is a class of central simple $K_l$-algebra represented by $A$ . And we know the composite map $Cor\circ Res:  Br(\OO_{K_l}[1/p])[p]\cong Br(\OO_{K_{l+1}}[1/p])[p]\rightarrow Br(\OO_{K_l}[1/p])[p]$ is the multiplication by $p$. Hence the composite map is $0$. Therefore, the correstriction map $ Br(\OO_{K_{l+1}}[1/p])[p]\rightarrow Br(\OO_{K_l}[1/p])[p]$ is zero map. And take inverse limit with respect to the correstriction map, we have $\varprojlim Br(\OO_{K_l}[1/p])[p]\cong Br(\OO_{K}[1/p])[p] \cong\F_p^{\#S-1} $.

We have the following exact sequence:
\begin{equation*}
   0\rightarrow p\Cl_S(K_l)  \rightarrow \Cl_S(K_l)\rightarrow  \Cl_S(K_l)/p\rightarrow 0 
\end{equation*}
Take inverse limit,
\begin{equation*}
   0\rightarrow pX_{cs}  \rightarrow X_{cs}\rightarrow  \varprojlim\Cl_S(K_l)/p\rightarrow 0 
\end{equation*}
Hence $ \varprojlim\Cl_S(K_l)/p= X_{cs}/p$. We have $\mu_{cs}=0$ if and only if $X_{cs}/p$ is a finite group if and only if $ \varprojlim H^2(G_{K_l,S},\mu_p)$ is a finite group if and only if there exists $k$ such that $\Psi^{(k)}\neq 0$ for some $k$ by lemma \ref{main lemma} . Now assume $\mu_{cs}=0$. It is well known that the $\Z_p$ rank of $ X_{cs}$ is $\lambda_{cs}$. Since we assume that $X_{cs}$ has no torsion element, we have $ \varprojlim\Cl_S(K_l)/p= X_{cs}/p\cong \F_p^{\lambda_{cs}}$.

Recall that $p$ cohomological dimension of $G_{K,S}$ is 2. We have \[\#\varprojlim H^2(G_{K_l,S},\mu_p)=\#H_{\Iw}^2(G_{K_\infty,S},\mu_p)=p^n,\] 
where $n=\min\{n|\Psi^{(n)}\neq 0\}$ by lemma \ref{main lemma}. 
Therefore, by the exact sequence, we have
\[
\lambda_{cs}- \min\{n|\Psi^{(n)}\neq 0\}+\#S-1=0
\]
Hence $ \lambda_{cs}= \min\{n|\Psi^{(n)}\neq 0\}-\#S+1$

\end{proof}
In the proof of the theorem, to use the lemma \ref{main lemma}, we take $G=G_{K,S}$, $N=G_{K_\infty,S}$, $ \chi: G\rightarrow G/N\cong \Z_p$ and $\Omega=\F_p[[G/N]]$. And these are the set up that we use for most of the section. The purpose of the setup is purely for theoretical consistency. In practice, to calculate the Massey product, we would like $\Omega=\F_p[[G/N]]$ to be small. If we know $\lambda_{cs}<p^l$ for some $l$ in advance, we can take $G=G_{K,S}$, $N=G_{K_{p^l},S}$, $ \chi: G\rightarrow G/N\cong \Z/p^l\Z$ and $\Omega=\F_p[[G/N]]$.

\begin{theorem}\label{main theom small}
Let $K\subset K_1\subset K_2\subset \cdots \subset K_\infty$ be a $\Z_p$ extension of $K$ and $S$ be the set of primes above $p$ for $K$. Assume all primes in $S$ are totally ramified in $K_\infty/K$. Let $X_{cs}=\varprojlim\Cl_S(K_l)$ and $\mu_{cs}$, $\lambda_{cs}$  be the Iwasawa invariant of $X_{cs} $. Assume $X_{cs} $ has no torsion element and $H^2(G_{K,S},\mu_p)\cong \F_p$ and $\mu_{cs}=0$. Assume $\lambda_{cs}<p^l $,  Then $\lambda_{cs}=\min\{n|\Psi^{(n)}\neq 0\}-\#S+1$. Here the definition of $\Psi^{(n)}$ is with respect to $G=G_{K,S}$, $N=G_{K_l,S}$, $ \chi: G\rightarrow G/N\cong \Z/p^l\Z$ and $\Omega=\F_p[[G/N]]$
    
\end{theorem}
\begin{proof}
 We have the following exact sequence:
    \begin{equation*}
   0\rightarrow \Cl_S(K_l)/p\rightarrow H^2(G_{K_l,S},\mu_p)\rightarrow Br(\OO_{K_l}[1/p])[p]\rightarrow0 
\end{equation*}
  We will use Lemma 13.15 in \cite{MR1421575} and the same notation as \cite{MR1421575}*{Lemma 13.15}. We have $\Cl_S(K_l)/p=X_{cs}/(\frac{(1+T)^{p^l}-1}{T}Y_0,pX_{cs})=X_{cs}/(T^{p^l-1}Y_0,pX_{cs})$. Let $f$ 
be the characteristic polynomial of $X_{cs}$. Since $X_{cs}$ has no torsion element, we have $fX_{cs}=0$. We have 
\begin{equation*}
    \begin{split}
        \Cl_S(K_l)/p&=X_{cs}/(T^{p^l-1}Y_0,pX_{cs})=X_{cs}/(T^{p^l-1}Y_0,pX_{cs},fX_{cs})\\&=X_{cs}/(T^{p^l-1}Y_0,pX_{cs},T^{\lambda_{cs}}X_{cs})=X_{cs}/(pX_{cs},T^{\lambda_{cs}}X_{cs})\\&=X_{cs}/(pX_{cs},fX_{cs})=X_{cs}/pX_{cs}\cong \F_p^{\lambda_{cs}}
    \end{split}
\end{equation*} since $\lambda_{cs}<p^l$. 

The remaining proof is similar to the proof of Theorem \ref{main theom}
\end{proof}

Next, we will discuss the situation when we have a group $\Delta$ acting on our Iwasawa module. We can decompose our Iwasawa module as a direct sum of eigenspace with respect to the action. For each direct sum, we can also define the Iwasawa $\lambda$ invariant. We will show that we can compute the Iwasawa $\lambda$ invariant in the same strategy since we have proved that the generalized Bockstein map preserves the group action in section \ref{generalized Bockstein map}.

Let $k$ be a number field and $K/k$ be an abelian extension. Denote $\Gal(K/k)=\Delta$.  Let $K\subset K_1\subset K_2\subset \cdots \subset K_\infty$ be a $\Z_p$ extension of $K$ and $K_\infty/k$ is an abelian extension. Suppose all the field extensions $K_l/k$, $K_\infty/k$, $K_S/k$ are Galois extensions. Assume $\Gal(K_\infty/k)\cong \Delta\oplus \Z_p$. Let $\G=G_{k,S}=\Gal(K_S/k)$, $G=G_{K,S}=\Gal(K_S/K)$, $N=G_{K_\infty,S}=\Gal(K_S/K_\infty)$ and $T=\mu_p$. By subsection \ref{generalized Bockstein map}, the Galois group $\Gal(K/k)=\Delta$ can act on $H^i(G_{K,S},\Omega/I^n\otimes \mu_p)$ and $\Psi^{(n)}(\tau \phi)=\tau \Psi^{(n)}(\phi)$ for any $\phi \in H^1(G,\Omega/I^n\otimes \mu_p)$ and $\tau \in \Delta$.

Let $\hat{\Delta}\coloneqq \Hom(\Delta, \Z_p)$ be the character group. Let $\omega \in \hat{\Delta}$ and define
\[
\varepsilon_{\omega}=\frac{1}{\# \Delta} \sum_{\sigma \in \Delta}\omega(\sigma)\sigma^{-1}\in \Z_p[\Delta].
\]
Let $X$ be any $\Z_p[\Delta]$ module. Then the standard process gives us a decomposition of $X$:
\[
X=\bigoplus_{\omega\in\hat{\Delta}} \varepsilon_{\omega}X.
\]

\begin{theorem}\label{action main theom}
  Let $K_\infty/K$ be the $\Z_p$ extension as set up above. Assume all primes in $S$ begin totally ramified starting $K$. Let $X_{cs}=\varprojlim\Cl_S(K_l)$. Assume $p\nmid \#\Delta $. The action of $\Delta$ on $ X_{cs}$ gives a decomposition $X_{cs}=\oplus_\omega \varepsilon_\omega X_{cs}$. Let $\mu_{\omega,cs}$, $\lambda_{\omega,cs}$  be the Iwasawa invariant of $\varepsilon_\omega X_{cs} $. Assume $\varepsilon_\omega X_{cs} $ has no torsion element and $\varepsilon_\omega H^2(G_{K,S},\mu_p)\cong \F_p$. Then  $\mu_{\omega,cs}=0$ if and only if there exist $k$ such that $\varepsilon_\omega\Psi^{(k)}\neq 0 $ for some $k$. If $\mu_{\omega,cs}=0$ then $\lambda_{\omega,cs}=\min\{n|\varepsilon_\omega\Psi^{(n)}\neq 0\}-dim_{\F_p}\varepsilon_\omega Br(\OO_{K}[1/p])[p]$  
\end{theorem}
\begin{proof}
  The proof is almost the same as Theorem \ref{main theom}, we give a sketch of proof here.

  We have the following exact sequence:
    \begin{equation*}
   0\rightarrow \Cl_S(K_l)/p  \rightarrow H^2(G_{K_l,S},\mu_p)\rightarrow Br(\OO_{K_l}[1/p])[p]\rightarrow 0 
\end{equation*}
for every $l$. The action of $\Delta$ gives us:
 \begin{equation*}
   0\rightarrow \varepsilon_\omega\Cl_S(K_l)/p  \rightarrow \varepsilon_\omega H^2(G_{K_l,S},\mu_p)\rightarrow \varepsilon_\omega Br(\OO_{K_l}[1/p])[p]\rightarrow 0 
\end{equation*}
Take inverse limit,
 \begin{equation*}
   0\rightarrow \varepsilon_\omega\varprojlim\Cl_S(K_l)/p  \rightarrow \varepsilon_\omega \varprojlim H^2(G_{K_l,S},\mu_p)\rightarrow \varepsilon_\omega \varprojlim Br(\OO_{K_l}[1/p])[p]\rightarrow 0 
\end{equation*}
We have $\mu_{\omega,cs}=0$ if and only if $ \varepsilon_\omega\varprojlim\Cl_S(K_l)/p$ is finite if and only if $\varepsilon_\omega \varprojlim H^2(G_{K_l,S},\mu_p)$ is finite if and only if there exist $k$ such that $\varepsilon_\omega\Psi^{(k)}\neq 0 $ for some $k$. 

Assume $\mu_{\omega,cs}=0 $ now. Similar argument as proof of Theorem \ref{main theom}, we have \[\varepsilon_\omega \varprojlim Br(\OO_{K_l}[1/p])[p]=\varepsilon_\omega  Br(\OO_{K}[1/p])[p]\]
and \[\varepsilon_\omega\varprojlim\Cl_S(K_l)/p\cong\F_p^{\lambda_{\omega,cs}}\]
and 
\[\varepsilon_\omega \varprojlim H^2(G_{K_l,S},\mu_p)=\min\{n|\Psi^{(n)}|_{\varepsilon_\omega H^2(G_{K,S},\Omega/I^n\otimes \mu_p)}\neq 0\},\] 
where $\Psi^{(n)}|_{\varepsilon_\omega H^2(G_{K,S},\Omega/I^n\otimes \mu_p)}$ denotes $ \Psi^{(n)}$ restricting on $\varepsilon_\omega H^2(G_{K,S},\Omega/I^n\otimes \mu_p)$. Since $\Psi^{(n)}(\tau \phi)=\tau \Psi^{(n)}(\phi)$ for any $\phi \in H^1(G,\Omega/I^n\otimes \mu_p)$ and $\tau\in \Delta$, we have $\Psi^{(n)}|_{\varepsilon_\omega H^2(G_{K,S},\Omega/I^n\otimes \mu_p)}=0$ if and only if $\varepsilon_\omega\Psi^{(n)}=0$. Hence the conclusion follows from the exact sequence. 
\end{proof}
\begin{remark}
    We can have a similar theorem as theorem \ref{main theom small} with the action of $\Delta$. We omit it here since it is essentially the same.
\end{remark}
Before we apply the theorem to a specific field, we first recall some well-known lemmas in the classical Iwasawa theory. They will be used in the next section. The next lemma is Proposition 13.26 in the book \cite{MR1421575}.
\begin{lemma}\label{injective map}
    Let $p$ be odd. Suppose $K$ is a CM-field and $K_\infty/K$ is the cyclotomic $\Z_p$ extension. The complex conjugation action gives us the decomposition $ \Cl(K_l)[p^\infty]=\Cl(K_l)[p^\infty]^+\oplus \Cl(K_l)[p^\infty]^-$. Assume all primes which are ramified in $K_\infty/K$ are totally ramified. Then the map
    \[\Cl(K_l)[p^\infty]^-\rightarrow \Cl(K_{l+1})[p^\infty]^-\]
    is injective.
\end{lemma}
By using the lemma \ref{injective map}, one can get the following lemma. It is Proposition 13.28 in \cite{MR1421575}.
\begin{lemma} \label{no finite mudule}
    The same assumption as Lemma \ref{injective map}, then $X^-=\varprojlim \Cl(K_l)[p^\infty]^-$ contains no finite $\Lambda$-submodules.
\end{lemma}
In the paper \cite{Gold}, Gold refined the Lemma \ref{injective map} to the following theorem.
\begin{theorem}\label{inj map}
    The same assumption as lemma \ref{injective map}, let $D_l$ be the subgroup of $\Cl(K_l)[p^\infty]$ generated by primes of $K_l$ above $p$. Hence we have $\Cl_S(K_l)[p^\infty]=\Cl(K_l)[p^\infty]/D_l$ and  $\Cl_S(K_l)[p^\infty]^-=\Cl(K_l)[p^\infty]^-/D_l^-$ by definition. Then the map
    \[
    \Cl_S(K_l)[p^\infty]^-\rightarrow \Cl_S(K_{l+1})[p^\infty]^-
    \]
    is injective. Moreover $\# D_m^-=p^{s(m-l)}\#D_l^-$ for  any integer $m\geq l$, where $s$ is the number of primes of $K_l^+$ lying over $p$ which split in  $K_l/K_l^+$.
\end{theorem}
The conclusion about the size of $D_l$ is hidden in Gold's proof of his theorem. 
\begin{theorem}\label{no finite}
    The same assumption as lemma \ref{injective map}, then $X_{cs}^-=\varprojlim \Cl_S(K_l)[p^\infty]^-$ contains no finite $\Lambda$-submodules.
\end{theorem}
\begin{proof}
    One can use a similar argument as the proof that lemma \ref{injective map} implies lemma \ref{no finite mudule} in \cite{MR1421575}. I omit the proof here since it is almost the same argument.
\end{proof}
\section{Application to concrete examples}
In this section, we will apply the theorems that we get in the previous section \ref{Applications to number theory} to some concrete examples. In the following examples, they all are the cyclotomic $\Z_p$ extension.   The same strategy can apply to other $\Z_p$ extensions as long as it satisfies the conditions in the theorem. The first three cases are about imaginary quadratic fields. The last case is about cyclotomic fields. In the first case of the imaginary quadratic field, we also develop a numerical criterion, which can help us to compute the $\lambda$ invariant through the computer. The section of numerical criterion is independent from other cases.  It is suggested to skip the numerical criterion subsection for the first time reading.
\subsection{Imaginary Quadratic field}

\subsubsection{case 1}\label{case2}
Let $K$ be an imaginary quadratic field and $h_K=\#Cl(K)$, $p\nmid h_K$, $p$ splits in $K$. Let $K\subset K_1\subset K_2\subset \cdots \subset K_\infty$ be the cyclotomic $\Z_p$ extension of $K$. Let $p\OO_{K_l}=\p_l\tilde{\p}_l$ where $\tilde{\p}_l$ is the complex conjugation of $\p_l$ . Let $(\alpha)= \p_0^{h_K}$, $(\tilde{\alpha})= \tilde{\p}_0^{h_K}$, $\Tilde{\alpha}$ is the conjugation of $\alpha$. By Kummer theory, $\alpha$ corresponds to an element in $H^1(G,\mu_p)$, i.e. $\sigma\rightarrow \sigma (\sqrt[p]{\alpha})/\sqrt[p]{\alpha}$. Recall $\chi$ is a character $\chi: G_{K,S}\xrightarrow{Res} \Gal(K_\infty/K)\cong \Z_p$. 

\begin{theorem}\label{main1}
Let $K$ be an imaginary quadratic field, $p\nmid h_K$, $p$ splits in $K$ and $n\geq 2$. Assume $\lambda\geq n-1$, then $\lambda\geq n\Leftrightarrow$ $n$-fold Massey product $(\chi,\chi,\cdots\chi,\alpha)$ vanishes with respect to a proper defining system. 

\end{theorem}
\begin{remark}
    It is well known that $\lambda\geq 1$ is always true. Hence $\lambda \geq 2\Leftrightarrow$ the cup product $\chi\cup \alpha=0$. Later, in the section on numerical criterion, we will show that it is easy to see that $\chi\cup \alpha=0 \Leftrightarrow \log_p\alpha \equiv 0 \mod{p^2}$, where $\log_p$ is the $p$-adic log. And this is a version of Gold's criterion \cite{Gold}. Our theorem gives a new proof of Gold's criterion and generalizes the criterion.
\end{remark}

\begin{proof}[proof of theorem \ref{main1}] 
\hfill
  
We have that $K$ is an imaginary quadratic field and $K_\infty/K$ is a cyclotomic $\Z_p$ extension. Hence, it satisfies the assumption in theorem \ref{inj map} and theorem \ref{no finite} and we will use the same notation as them. By using Proposition 13.22 in \cite{MR1421575}, we know $\Cl(K_l)[p^\infty]^+=0$. Hence $D_l^-=D_l$ and $\Cl(K_l)[p^\infty]^-=\Cl(K_l)[p^\infty]$. By definition, we have that $D_l^-$ is generated by one element $ \p_l/\tilde{\p}_l$. Hence $D_l=D_l^-$ is a cyclic group. There is only one prime in $K_l^+$ lying over $p$ that splits in $K_l/K_l^+$. Thus $\#D_l= p^l\#D_0=p^l$. We have $D_l=\Z/p^l\Z$. We have the following exact sequence:
\[
0\rightarrow D_l\rightarrow Cl(K_l)[p^\infty]\rightarrow  Cl_S(K_l)[p^\infty]\rightarrow 0
\]
Take the inverse limit, we have
\[
0\rightarrow \Z_p \rightarrow X \rightarrow  X_{cs}\rightarrow 0
\]
By lemma \ref{no finite mudule} and theorem \ref{no finite}, we know that $X$ and $X_{cs}$ have no finite $\Lambda$-module. View them as $\Z_p$ module, we have $Rank_{\Z_p}\Z_p-Rank_{\Z_p}X+Rank_{\Z_p}X_{cs}=0$. So $\lambda=\lambda_{cs}+1$.

 By the exact sequence \eqref{2}, we have
 \[  0\rightarrow \Cl_S(K)/p =0 \rightarrow H^2(G_{K,S},\mu_p)\rightarrow Br(\OO_K[1/p])[p]=\F_p \rightarrow 0 
\]
Hence $H^2(G_{K,S},\mu_p)\cong F_p$. Thus, our case satisfies all conditions in the theorem \ref{main theom}. We have $\lambda_{cs}=\min\{n|\Psi^{(n)}\neq 0\}-2+1$. Therefore, we have $\lambda=\min\{n|\Psi^{(n)}\neq 0\}$. Hence $\lambda \geq n \Leftrightarrow \Psi^{(i)}=0$ for all $0\leq i< n $.
 
   By exact sequence \eqref{1} and $\Cl(K)[p]=0$,   $H^1(G_{K,S},\mu_p)\cong \OO_{K_S}^*/p=<\alpha,\tilde{\alpha}>$. Assume $\Psi^{(i)}=0 $ for all $ 0\leq i< n-1$ now, then by theorem \ref{reduce}, $\Ima\Psi^{(n-1)}$ is generated by $\Psi^{(n-1)}([\alpha]) $ and $\Psi^{(n-1)}([\tilde{\alpha}])$. 
    
   Let $\zeta_{p^l}$ be the primitive $p^l$-th root of unit. Let $\Q_{l-1}$ be the unique subfield of $\Q(\mu_{p^l})$ such that $\Gal(\Q_{l-1}/\Q)=\Z/{p^{l-1}\Z}$. Then $K_l=K\Q_l$. We have \[p=\Nrm_{\Q(\mu_{p^n})/\Q}(1-\zeta_{p^l})\] 
   and \[1-\zeta_{p^l}=\Nrm_{\Q(\mu_{p^{l+1}})/\Q(\mu_{p^{l}})}(1-\zeta_{p^{l+1}}).\]
     Let \[\eta_l\coloneqq \Nrm_{\Q(\mu_{p^{l+1}})/\Q_l}(1-\zeta_{p^{l+1}}),\] then \[p=\Nrm_{\Q_l/\Q}(\eta_l) ,\eta_l=\Nrm_{\Q_{l+1}/\Q_l}(\eta_{l+1}).\]  
     Since $K_l=K\Q_l$ and $K\cap \Q_l=\Q$, then \[\Gal(K_l/K)=\Gal(\Q_l/\Q),\Gal(K_{l+1}/K_l)=\Gal(\Q_{l+1}/\Q_l).\] So $p=\Nrm_{K_l/K}(\eta_l)$ and $\eta_l=\Nrm_{K_{l+1}/K_l}(\eta_{l+1})$. We know \[\eta_l=\Nrm_{\Q(\mu_{p^{l+1}})/\Q_l}(1-\zeta_{p^{l+1}})\in \OO_{\Q_l,S}\subset \OO_{K_l,S}.\] Hence the sequence $(\eta_l)_l \in \varprojlim \OO_{K_l,S}$. By theorem \ref{cv in NT}, we know $p$ has $p$ cyclic Massey product vanishing property. 

   On the other hand, $\alpha \Tilde{\alpha}=\pm p^{h_K}$. Hence,\[\Psi^{(n-1)}([\alpha])+\Psi^{(n-1)}([\tilde{\alpha}])=h_K \Psi^{(n-1)}([\pm p])=0.\] So $\Ima\Psi^{(n-1)}$ is generated by $\Psi^{(n-1)}([\alpha]) $. 
   
   We have $\Psi^{(n-1)}=0$ if and only if $\Psi^{(n-1)}([\alpha])=0 $ i.e. the Massey product $(\chi^{(n-1)},\alpha)$ relative to a proper defining system vanishes by theorem \ref{LLSWW}.

\end{proof}

\begin{remark}\label{small remark}
        Suppose we know $\lambda<p^l$ in advance. Assume $\lambda\geq n-1$, then $\lambda\geq n\Leftrightarrow$ $n$-fold Massey product $(\chi,\chi,\cdots\chi,\alpha)=0$ with respect to a proper defining system, where we could take $\chi: G\rightarrow \Gal(K_l/K)\cong \Z/p^l\Z$ instead of $\chi: G\rightarrow \Gal(K_\infty/K)\cong \Z_p$ as above. The proof is similar by using Theorem \ref{main theom small} instead of Theorem \ref{main theom}. We will use this idea to develop numerical criterion.
\end{remark}

\subsubsection{numerical criterion}\label{numerical criterion1}

In this subsection, we will translate the Massey product language in the previous case \ref{case2} to numerical criterion. One is suggested to skip this subsection for the first time reading. It does not influence the reading for other cases.

Let $K$ be an imaginary quadratic field, $p\nmid h_K$, $p$ splits in $K$. Assume we know $\lambda<p$ in advance. By the remark \ref{small remark}, we can take $G=G_{K,S}$, $N=G_{K_1,S}$, $\chi: G\rightarrow G/N\cong \F_p$, $\Omega=\F_p[G/N]$ through this subsection. Now, we can state our numerical criterion: 
\begin{enumerate}
    \item The Iwasawa invariant $\lambda\geq1$, always true.
    \item The Iwasawa invariant $\lambda \geq 2$ $\Leftrightarrow \log_p\alpha \equiv 0 \mod{p^2}$, where $\log_p$ is the $p$-adic log.

   \item  Assume $\lambda\geq 2$ is true, then $ \chi\cup \alpha=0\Longleftrightarrow \exists \beta \in K^*_1$ s.t. $\Nrm_{K_1/K}(\beta)=\alpha$. Define $A_1'=\prod_{i=0}^{p-1}\sigma^i(\beta^i)\in K_1^*$ , where $\sigma$ is the generator of the group $G/N=\Gal(K_1/K)\cong \F_p$ such that $\chi(\sigma)=1$.

   Claim: There exists $ \alpha_1 \in K^*$ s.t. $ v_{\p}(\alpha_1 A_1')\equiv 0 \mod{ p} $ for all primes $\p$ in $K$ such that $ \p\nmid p$, where $ v_{\p}$ is the valuation corresponding to the prime ideal $\p$.

   Then $\lambda \geq 3\Longleftrightarrow \chi\cup \alpha_1=0\Longleftrightarrow \log_p \alpha_1 \equiv 0 \mod{p^2}$.

   \item Assume $\lambda \geq 3$, then $\chi \cup \alpha_1=0\Longleftrightarrow \exists \beta_1 \in K^*_1$ s.t. $\Nrm_{K_1/K}(\beta_1)=\alpha_1$,  Define $A_2'=\prod_{i=0}^{p-1}\sigma^{i}(\beta^{i(i-1)/2})\in K_1^*$ and $B_1'=\prod_{i=0}^{p-1}\sigma^i(\beta_1^i)\in K_1^*$.

   Claim: There exists $ \alpha_2 \in K^*$ s.t. $ v_{\p}(\alpha_2 A_2'B_1')\equiv 0 \mod{ p} $ for all primes $\p$ in $K$ such that $ \p\nmid p$.

   Then $\lambda \geq 4\Longleftrightarrow \chi\cup \alpha_2=0\Longleftrightarrow \log_p \alpha_2 \equiv 0 \mod{p^2}$.

   \item continues in a similar way $\cdots$

\end{enumerate}

In the remaining subsection, we will explain the reason for the numerical criterion. Here is the idea. To prove the Massey product relative to the proper defining system $\Bar
{\rho}_n:G_K\rightarrow\Bar{U}_{n+1}(\F_p),n\leq p$ vanishes, 
\[\Bar{\rho}_n=\begin{bmatrix}  
    1 & \chi &\binom{\chi}{2}&\binom{\chi}{3}&\binom{\chi}{4}&\cdots & *\\ 
    0 &1& \chi &\binom{\chi}{2}&\binom{\chi}{3}&\cdots & \psi_{n-2}\\
    \vdots& \vdots& \vdots& \vdots& \vdots& \ddots& \vdots\\
   0&0&0& 1&\chi&\binom{\chi}{2}&\psi_2\\
   0&0&0&0&1&\chi&\psi_1\\
   0&0&0&0&0&1&\psi_0\\
   0&0&0&0&0&0&1
    \end{bmatrix}
\]
we would like to construct the cochain $\psi_{n-1}$ explicitly to fill the $*$ spot. We first construct $\psi_{n-1}'$ when we view the Massey product over $G_K=\Gal(K^{sep}/K) $ instead of $G_{K,S}$. In this case, it does not matter if we assume that $K$ has $p$-th root of units. Hence, we can exploit the classical results and generalize them. Then we compare the Massey product when we view it over $G_K$ and $G_{K,S}$ differently. The difference of $ \psi_{n-1}$ and $\psi_{n-1}'$ is an element $\alpha_{n-1}$ in $K$. Finally, to check the Massey product is zero over $G_{K,S}$, we restrict the Massey product to $G_{K_{\p_0}}$, where $ K_{\p_0}$ is the completion of $K$ at prime $\p_0$.

Next, we introduce some classical results. Let $K$ be any field (not necessarily a number field) such that $char(K)\neq p$. Let $G_K=\Gal(K^{sep}/K)$ be the absolute Galois group. Let $\zeta_p$ be the $p$-th root of the unit. Then $[K(\zeta_p):K]$ has degree which is prime to $p$. Hence 
\[Cor\circ Res: H^i(G_K,\mu_p)\rightarrow H^i(G_{K(\zeta_p)},\mu_p)\rightarrow  H^i(G_K,\mu_p)\]
is an isomorphism. Therefore, $ Res: H^i(G_K,\mu_p)\rightarrow H^i(G_{K(\zeta_p)},\mu_p)$ is an injective map. To consider the Massey product relative to certain defining system vanishing in $H^2(G_K,\mu_p)$, we just need to restrict all cochains to $G_{K(\zeta_p)}$ and consider the Massey product vanishing inside $H^2(G_{K(\zeta_p)},\mu_p)$. It does not matter if we assume $\zeta_p\in K$. Now we can assume that $\zeta_p\in K$. In other words, the action $G_K$ on $\mu_p$ is trivial. We can view $\mu_p$ as $\F_p$ by identifying $\zeta_p\in 
\mu_p$ with $1\in \F_p$.

As mentioned before, to give a defining system is the same to give a homomorphism $\bar{\rho}: G_K\rightarrow \bar{U}_{n+1}(\mu_p)=\bar{U}_{n+1}(\F_p)$. The Massey product vanishing is equivalent to that $ \bar{\rho}$ can be lifted to a homomorphism $ \rho: G_K\rightarrow U_{n+1}(\mu_p)$. Let $K_{\bar{\rho}}$ be the subfield fixed by $\ker\bar{\rho}$. Then the Massey product $ (\chi_1,\chi_2,\cdots,\chi_n)_{\bar{\rho}}$ vanishing is equivalent to that we can extend the Galois extension $K_{\bar{\rho}}/K $ to a Galois extension $ K_{\rho}/K$ such that 
\[
\begin{tikzcd}
\Gal(K_{\rho}/K) \arrow[r, hook] \arrow[d, two heads]
& U_{n+1}(\mu_p)  \arrow[d, two heads] \\
\Gal(K_{\bar{\rho}}/K)  \arrow[r, hook]
&   \bar{U}_{n+1}(\mu_p)
\end{tikzcd}
\]
is commutative.

To understand this better, we first consider 2-fold Massey products i.e. cup products. Let $a,b$ be two linearly independent elements in $F^*/(F^*)^p$ corresponding to $\chi_a,\chi_b\in H^1(G_K,\mu_p)\cong F^*/(F^*)^p$. Then $K(a^{1/p})$ is the subfield fixed by group $\ker\chi_a$ and $K(b^{1/p})$ is the subfield fixed by group of $\ker\chi_b$. The cup product $ \chi_a\cup \chi_b=0$ is equivalent that $\exists \beta \in K(a^{1/p})$ such that $ \Nrm_{K(a^{1/p})/K}(\beta)=b$. Translating in the language of Massey product, we have a defining system 
\[\Bar{\rho}: G_K\rightarrow \Bar{U}_3(\mu_p)\cong \F_p\oplus \F_p, \sigma\rightarrow (\chi_a(\sigma),\chi_b(\sigma))\]
The 2-fold Massey product vanishing is equivalent to that we can extend the Galois extension $K(a^{1/p},b^{1/p})/K$ to a Heisenberg extension $K(a^{1/p},b^{1/p},A^{1/p})/K$ such that $\Gal(K(a^{1/p},b^{1/p},A^{1/p})/K)\cong U_3(\F_p)$. This is a theorem proved by Romyar Sharifi\cite{MR2716836}. Another reference is \cite{MR3614934}:

\begin{theorem}[Sharifi]
    The same notation as before. Assume $\chi_a\cup\chi_b=0 $. Let $A_1=\prod_{i=0}^{p-1}\sigma^j(\beta^j)$ where $\sigma$ is a generator of $\Gal(K(a^{1/p}/K)$ such that $\chi(\sigma)=1$, then $ \sigma(A_1)=A_1\frac{\beta^p}{b}$.
\end{theorem}

\begin{theorem}[Sharifi]\label{sharifi explicit}
    The same notation as before. Assume $\chi_a\cup\chi_b=0$. Let $A=fA_1$ where $f\in K^*$, then the homomorphism \[\Bar{\rho}: G_K\rightarrow \Bar{U}_3(\mu_p)\cong \F_p\oplus \F_p, \sigma\rightarrow (\chi_a(\sigma),\chi_b(\sigma))\]
    can be lifted to a Heisenberg extension $\rho:G_K\rightarrow U_3(\mu_p)$ such that $Res_{\ker\chi_a}(\rho_{1,3})=\chi_A $. 
\end{theorem}
\begin{remark}

   When we have another lifting corresponding to $A'$, the difference between $A$ and $A'$ is an element in $K^*$. 
   If we have another lifting $\rho'$, then $-d\rho_{1,3}=\chi_a\cup\chi_b$ and $-d\rho'_{1,3}=\chi_a\cup\chi_b$. So $d(\rho_{1,3}-\rho'_{1,3})=0$. Hence $\rho_{1,3}-\rho'_{1,3}$ is a cocycle in $\mathcal{C}^1$. There exists $f\in K$ such that $\rho_{1,3}-\rho'_{1,3}=\chi_f$. We have $A=fA'$ up to multiplication by an element of $K(a^{1/p})^{*p}$.

    The theorem gives us an explicit description of $\rho_{1,3}$. And the converse is also true. If we can find such Heisenberg extension $ K(a^{1/p},b^{1/p},A^{1/p})/K$ then the cup product $\chi_a\cup\chi_b$ is zero.
    
\end{remark}

Now, we generalize the idea to give a similar description for Massey product $(\chi^{(n)},\psi_0)$ relative to a proper defining system $\bar{\rho}_{n+1}: G_K\rightarrow \bar{U}_{n+2}(\F_p)$.

Because we use the proper defining system, the $\Ima \bar{\rho}_{n+1} $ is not the whole group $\bar{U}_{n+2}(\F_p)$. We need the following definition and lemma to describe $\Ima \bar{\rho}_{n+1}$.
\begin{definition}
    Define group $M_n:=<s,t_0,t_1,\cdots, t_n\mid s^p=1, t_0^p=1,t_1^p=1,\cdots, t_n^p=1, t_it_jt_i^{-1}t_j^{-1}=1 ,st_0s^{-1}t_0^{-1}=t_1, st_1s^{-1}t_1^{-1}=t_2,st_2s^{-1}t_2^{-1}=t_3,\cdots ,st_{n-1}s^{-1}t_{n-1}^{-1}=t_n,st_ns^{-1}t_n^{-1}=1 >$
\end{definition}
\begin{remark}
    One can check that $M_n$ is a semiproduct. We have $M_n\cong \F_p\ltimes \F_p^{n+1}$, where $\F_p=<s>$ and $\F_p^{n+1}=<t_0,t_1,\cdots,t_n>$.
\end{remark}
\begin{lemma}\label{structure of image}
    Let $\rho_{n+1}: G_K\rightarrow U_{n+2}(\F_p)$ be the homomorphism defined by
    
     \[ \rho_{n+1}=\begin{bmatrix}  
    1 & \chi &\binom{\chi}{2}&\binom{\chi}{3}&\binom{\chi}{4}&\cdots & \psi_n\\ 
    0 &1& \chi &\binom{\chi}{2}&\binom{\chi}{3}&\cdots & \psi_{n-1}\\
    \vdots& \vdots& \vdots& \vdots& \vdots& \ddots& \vdots\\
   0&0&0& 1&\chi&\binom{\chi}{2}&\psi_2\\
   0&0&0&0&1&\chi&\psi_1\\
   0&0&0&0&0&1&\psi_0\\
   0&0&0&0&0&0&1
    \end{bmatrix}\]
     where $\ker\chi\neq \ker\psi_0$.  
   Then $\Ima\rho_{n+1}$ is isomorphic to $M_n$. If we view $\Ima\rho_{n+1} \subset U_{n+2}(\F_p) $ as a subgroup, then we can take
   \[ s=\begin{bmatrix}  
    1 & 1 &0&0&0&\cdots & 0\\ 
    0 &1& 1 &0&0&0 & 0\\
    \vdots& \vdots& \vdots& \vdots& \vdots& \ddots& \vdots\\
   0&0&0& 1&1&0&0\\
   0&0&0&0&1&1&0\\
   0&0&0&0&0&1&0\\
   0&0&0&0&0&0&1
    \end{bmatrix},
        t_0=\begin{bmatrix}  
    1 & 0 &0&0&0&\cdots & 0\\ 
    0 &1& 0 &0&0&0 & 0\\
    \vdots& \vdots& \vdots& \vdots& \vdots& \ddots& \vdots\\
   0&0&0& 1&0&0&0\\
   0&0&0&0&1&0&0\\
   0&0&0&0&0&1&1\\
   0&0&0&0&0&0&1
    \end{bmatrix},\]
\[
     t_1=\begin{bmatrix}  
    1 & 0 &0&0&0&\cdots & 0\\ 
    0 &1& 0 &0&0&0 & 0\\
    \vdots& \vdots& \vdots& \vdots& \vdots& \ddots& \vdots\\
   0&0&0& 1&0&0&0\\
   0&0&0&0&1&0&1\\
   0&0&0&0&0&1&0\\
   0&0&0&0&0&0&1
    \end{bmatrix},\cdots,
     t_n=\begin{bmatrix}  
    1 & 0 &0&0&0&\cdots & 1\\ 
    0 &1& 0 &0&0&0 & 0\\
    \vdots& \vdots& \vdots& \vdots& \vdots& \ddots& \vdots\\
   0&0&0& 1&0&0&0\\
   0&0&0&0&1&0&0\\
   0&0&0&0&0&1&0\\
   0&0&0&0&0&0&1
    \end{bmatrix}
    \]
\end{lemma}
\begin{proof}
    Since $\ker\chi\neq \ker\psi_0$, $[G_K:\ker\chi]=p $ and $ [G_K:\ker\psi_0]=p$ , so  $ \ker\chi$ and $\ker\psi_0$ do not contain each other.
    Hence $\Ima \rho_{n+1}$ contains 
    \[ A=\begin{bmatrix}  
    1 & 1 &0&0&0&\cdots & *\\ 
    0 &1& 1 &0&0&0 & *\\
    \vdots& \vdots& \vdots& \vdots& \vdots& \ddots& \vdots\\
   0&0&0& 1&1&0&*\\
   0&0&0&0&1&1&*\\
   0&0&0&0&0&1&0\\
   0&0&0&0&0&0&1
    \end{bmatrix},
        B_0=\begin{bmatrix}  
    1 & 0 &0&0&0&\cdots & *\\ 
    0 &1& 0 &0&0&0 & *\\
    \vdots& \vdots& \vdots& \vdots& \vdots& \ddots& \vdots\\
   0&0&0& 1&0&0&*\\
   0&0&0&0&1&0&*\\
   0&0&0&0&0&1&1\\
   0&0&0&0&0&0&1
    \end{bmatrix},\] 
One can define $B_1=AB_0A^{-1}B_0^{-1},\cdots, B_n=AB_{n-1}A^{-1}B_{n-1}^{-1}$. We can directly map $s$ to $A$ and $t_0$ to $B_0$. This gives us an isomorphism between $M_n$ and $\rho_{n+1}$. To check the group relations, it is a trivial calculation of matrices. We omit the calculation here.

On the other hand, we can use $A, B_0$ generating the following two matrices:
\[
     t_1=\begin{bmatrix}  
    1 & 0 &0&0&0&\cdots & 0\\ 
    0 &1& 0 &0&0&0 & 0\\
    \vdots& \vdots& \vdots& \vdots& \vdots& \ddots& \vdots\\
   0&0&0& 1&0&0&0\\
   0&0&0&0&1&0&1\\
   0&0&0&0&0&1&0\\
   0&0&0&0&0&0&1
    \end{bmatrix},\cdots,
     t_n=\begin{bmatrix}  
    1 & 0 &0&0&0&\cdots & 1\\ 
    0 &1& 0 &0&0&0 & 0\\
    \vdots& \vdots& \vdots& \vdots& \vdots& \ddots& \vdots\\
   0&0&0& 1&0&0&0\\
   0&0&0&0&1&0&0\\
   0&0&0&0&0&1&0\\
   0&0&0&0&0&0&1
    \end{bmatrix}
    \]
  This gives us the description of $\Ima \rho_{n+1}$ as in the lemma.  
\end{proof}
  \begin{remark}
     If we take subgroup $V_{n+1} \subset  U_{n+2}(\F_p)$ whose elements are of the form:
    \[ \begin{bmatrix}  
    1 & 0 &0&0&0&\cdots & *\\ 
    0 &1& 0 &0&0&0 & *\\
    \vdots& \vdots& \vdots& \vdots& \vdots& \ddots& \vdots\\
   0&0&0& 1&0&0&*\\
   0&0&0&0&1&0&*\\
   0&0&0&0&0&1&*\\
   0&0&0&0&0&0&1
    \end{bmatrix}
    \]
then we have an exact sequence, 
\[ 0\rightarrow V_{n+1}\rightarrow \Ima\rho_{n+1}\rightarrow \Ima\rho_{n+1}/V_{n+1}\rightarrow 0\]
Easy to see that $V_{n+1}\cong \F_p^{n+1}$ is generated by $t_0,t_1,\cdots, t_n$ and $\Ima\rho_{n+1}/V_{n+1}\cong \F_p $ is generated by $s$. And $\Ima\rho_{n+1}\cong \F_p\ltimes V_{n+1} $. Hence, to describe the presentation of $\Ima\rho_{n+1} $, all we need is to describe the presentations of $V_{n+1}$ and $\F_p $ and the action of $ \F_p$ on $V_{n+1}$. These are how we define $M_n$.  
  \end{remark}
  
\begin{remark}\label{s,t}
    By the lemma \ref{structure of image}, if we can find a Galois extension $L/K$ such that $\Gal(L/K)\cong M_n$, then we can determine a proper defining system $\Bar{\rho}_{n+2}: G_K\rightarrow \Bar{U}_{n+3} $. And the converse is also true. 
    
    To determine whether a group is isomorphic to $M_n$, we only need to determine what elements are mapped to $s, t_0$ and check all relations if we know the group has the same size as $M_n$.
\end{remark}
As before, Assume $K$ is a field that contains $p$-th root of the unit. Let $\chi\in H^1(G_K,\mu_p)\cong \Hom(G_K,\F_p)$ and $\psi_0\in H^1(G_K,\mu_p)\cong \Hom(G_K,\F_p)$. Let $K(a^{1/p})$ be the fixed subfield of $\ker\chi$ and $K(b^{1/p})$ be the fixed subfield of $\ker\psi_0$ and assume they are different field.  Let $\sigma_a$ and $\sigma_b $ be the generator of $\Gal(K(a^{1/p},b^{1/p})/K)$ such that
\begin{align*}
   \sigma_a(a^{1/p})&=\zeta_p a^{1/p}& \sigma_a(b^{1/p})&=b^{1/p} \\
   \sigma_b(a^{1/p})&= a^{1/p}& \sigma_b(b^{1/p})&=\zeta_p b^{1/p}
\end{align*}

\begin{lemma}\label{A_n}
 Assume $\chi\cup \psi_0=0$,i.e. there exists $\beta\in K(a^{1/p})$ such that $\Nrm_{K(a^{1/p})/K}(\beta)=b$. Define \[A_1=\prod_{i=0}^{p-1}\sigma_a^i(\beta^i), A_2=\prod_{i=0}^{p-1}\sigma_a^i(\beta^{\frac{i(i-1)}{2}}), A_3=\prod_{i=0}^{p-1}\sigma_a^i(\beta^{\frac{i(i-1)(i-2)}{3}}), \cdots , A
_n=\prod_{i=0}^{p-1}\sigma_a^i(\beta^{\binom{i}{n}})\] where $n<p$, 
Then
\[
\frac{\sigma_a(A_1)}{A_1}b=\beta^p, \frac{\sigma_a(A_2)}{A_2}\sigma_a(A_1)=\beta^{\frac{p(p-1)}{2}}
, \cdots , \frac{\sigma_a(A_n)}{A_n}\sigma_a(A_{n-1})=\beta^{\binom{p}{n}}
\]
\end{lemma}
\begin{proof}
    Easy to check the following equation:
    \[(\sigma_a-1)(\sum_{i=0}^{p-1}i\sigma_a^i) +\sum_{i=0}^{p-1}\sigma_a^i=p \]
    \[
    (\sigma_a-1)(\sum_{i=0}^{p-1}\frac{i(i-1)}{2}\sigma_a^i) +\sigma_a\sum_{i=0}^{p-1}i\sigma_a^i=\frac{p(p-1)}{2}
    \]
    \[\cdots\]
   \[(\sigma_a-1)(\sum_{i=0}^{p-1}\binom{i}{n}\sigma_a^i) +\sigma_a\sum_{i=0}^{p-1}\binom{i}{n-1}\sigma_a^i=\binom{p}{n}   \]
\end{proof}

\begin{lemma}\label{standard}
    Notations as before and $n<p$, then the field extension \[K(a^{1/p},b^{1/p},A_1^{1/p},A_2^{1/p},\cdots, A_n^{1/p})/K\] is a Galois extension. And \[\Gal(K(a^{1/p},b^{1/p},A_1^{1/p},A_2^{1/p},\cdots, A_n^{1/p})/K)\cong M_n\]. Therefore, it corresponds a proper defining system $\Bar{\rho}_{n+2}:G_K\rightarrow \Bar{U}_{n+3}$. And $Res_{\ker\chi}\psi_i=\chi_{A_i}$ for $1\leq i\leq n<p-1$.
\end{lemma}
\begin{proof}

   We have \[\sigma_a(A_1)/A_1=\beta^p/b\in K(a^{1/p},b^{1/p})^{*p}\] and \[ \sigma_b(A_1)/A_1=1\in K(a^{1/p},b^{1/p})^{*p}.\] Hence $K(a^{1/p},b^{1/p},A_1^{1/p})/K$ is a Galois extension.
   
    Lift $\sigma_a$ and $\sigma_b$ to $\Gal(K(a^{1/p},b^{1/p},A_1^{1/p})/K)$. We denote them as $\Tilde{\sigma}_a$ and $\Tilde{\sigma}_b$. 
   
   As said in remark \ref{s,t}, to prove $\Gal(K(a^{1/p},b^{1/p},A_1^{1/p})/K)\cong M_1$, we only need to determine what is mapped to $s,t_0$ and check that it satisfies the relations in the presentation of group $M_1$. Lift $\sigma_a$ and $\sigma_b$ to $\Gal(K(a^{1/p},b^{1/p},A_1^{1/p})/K)$. We denote them as $\Tilde{\sigma}_a$ and $\Tilde{\sigma}_b$. We would like to map $\Tilde{\sigma}_a$, $\Tilde{\sigma}_b$ to $s,t_0$ respectively. Next, we will check it satisfies the relations.

  By lemma \ref{A_n}, \begin{equation}\label{A_1}    \Tilde{\sigma}_a(A_1^{1/p})=\zeta_p^{j_1}A_1^{1/p}\beta/b^{1/p}
  \end{equation}
  for some $j_1\in \F_p$. Then 
  \[  \Tilde{\sigma}_a^2(A_1^{1/p})=\zeta_p^{j_1}A_1^{1/p}\beta/b^{1/p}\zeta_p^{j_1}\Tilde{\sigma}_a(\beta)/b^{1/p}
  \]
  \[\cdots\]
  \[ \Tilde{\sigma}_a^p(A_1^{1/p})=\zeta_p^{j_1 p}A_1^{1/p} \Nrm(\beta)/b=A_1^{1/p} \]
Hence $\Tilde{\sigma}_a^p=1$.
 
   We have $\Tilde{\sigma}_b(A_1^{1/p})=\zeta_p^{i_1}A_1^{1/p} $ for some $i_1\in \F_p$. Hence $\Tilde{\sigma}_b^p(A_1^{1/p})=A_1^{1/p}$. So $\Tilde{\sigma}_b^p=1$ .

We have 
\[\Tilde{\sigma}_a\Tilde{\sigma}_b(A_1^{1/p})=\Tilde{\sigma}_a(\zeta_p^iA_1^{1/p})=\zeta_p^{i_1+j_1}A_1^{1/p}\beta/b^{1/p}\]
   \[   \Tilde{\sigma}_b\Tilde{\sigma}_a(A_1^{1/p})=\Tilde{\sigma}_b(\zeta_p^{j_1}A_1^{1/p}\beta/b^{1/p})=\zeta_p^{j_1+i_1-1}A_1^{1/p}\beta/b^{1/p}
   \]
   \[ \Tilde{\sigma}_a\Tilde{\sigma}_b\Tilde{\sigma}_a^{-1}\Tilde{\sigma}_b^{-1}(A_1^{1/p})=\zeta_p A_1^{1/p}
   \]
Define $\Tilde{\sigma}_{A_1}=\Tilde{\sigma}_a\Tilde{\sigma}_b\Tilde{\sigma}_a^{-1}\Tilde{\sigma}_b^{-1}$. We have 
\[
\Tilde{\sigma}_{A_1}^p=1, \Tilde{\sigma}_a\Tilde{\sigma}_{A_1}\Tilde{\sigma}_a^{-1}\Tilde{\sigma}_{A_1}^{-1}=1,\Tilde{\sigma}_b\Tilde{\sigma}_{A_1}\Tilde{\sigma}_b^{-1}\Tilde{\sigma}_{A_1}^{-1}=1
\]
 Hence, $ \Tilde{\sigma}_{A_1}$  plays the role $t_1$ in the presentation of $M_1$ and we checked that the relations are satisfied. Therefore, we have $\Gal(K(a^{1/p},b^{1/p},A_1^{1/p})/K)\cong M_1$. And $G_K\rightarrow \Gal(K(a^{1/p},b^{1/p},A_1^{1/p})/K)\cong M_1\subset U_3(\F_p)$ gives us $ Res_{\ker\chi}\psi_1=\chi_{A_1}$.

 For next step: Similarly, we have \[\Tilde{\sigma}_{a}(A_2)/A_2=\beta^{p(p-1)/2}/\Tilde{\sigma}_{a}(A_1)=\beta^{p(p-1)/2}b/(A_1\beta^p)\in K(a^{1/p},b^{1/p},A_1^{1/p})^{*p}\] and \[ \Tilde{\sigma}_{b}(A_2)/A_2=1 \in K(a^{1/p},b^{1/p},A_1^{1/p})^{*p}.\] And we know from last paragraph, $\Gal(K(a^{1/p},b^{1/p},A_1^{1/p})^{*p})/K)$ is generated by $ \Tilde{\sigma}_{a},\Tilde{\sigma}_{b}$. Hence for any $\sigma\in \Gal(K(a^{1/p},b^{1/p},A_1^{1/p})^{*p}/K)$, we have $\sigma(A_2)/A_2\in K(a^{1/p},b^{1/p},A_1^{1/p})^{*p}$. Therefore, $K(a^{1/p},b^{1/p},A_1^{1/p},A_2^{1/p})/K$ is a Galois extension.

Lift $ \Tilde{\sigma}_{a},\Tilde{\sigma}_{b},\Tilde{\sigma}_{A_1}$ to $\Gal(K(a^{1/p},b^{1/p},A_1^{1/p},A_2^{1/p})/K)$. And we still denote the lifting as $ \Tilde{\sigma}_{a},\Tilde{\sigma}_{b},\Tilde{\sigma}_{A_1}$ by a little abusing notation. Similarly, we will check that $\Tilde{\sigma}_{a},\Tilde{\sigma}_{b} $ are mapped to $s, t_0\in M_2$ and they satisfy the relations in the definition of $M_2$.

By lemma \ref{A_n}
\[\Tilde{\sigma}_{a}(A_2^{1/p})=\zeta_p^{j_2}A_2^{1/p}b^{1/p}\beta^{\frac{p-3}{2}}/A_1^{1/p}\]
 for some $j_2\in \F_p$. By equation \eqref{A_1}, we have 
 \[
 \prod_{i=0}^{p-1}\Tilde{\sigma}_{a}^i(A_1^{1/p})=A_1\prod_{i=0}^{p-1}\Tilde{\sigma}_{a}^i(\beta^{p-1-i})/b^{\frac{p-1}{2}}=b^{\frac{p-1}{2}}
 \]
 \[
 \Tilde{\sigma}_{a}^p(A_2^{1/p})=A_2^{1/p}b \prod_{i=0}^{p-1}\Tilde{\sigma}_{a}^i(\beta^{\frac{p-3}{2}})/ \prod_{i=0}^{p-1}\Tilde{\sigma}_{a}^i(A_1^{1/p})=A_2^{1/p}
 \]

Hence $\Tilde{\sigma}_{a}^p=1$.

We have $\Tilde{\sigma}_{b}(A_2)=A_2$, so $\Tilde{\sigma}_{b}(A_2^{1/p})=\zeta_p^{i_2}A_2^{1/p}$ for some $i_2\in \F_p$. So $\Tilde{\sigma}_{b}^p=1$

We have 
\[
\Tilde{\sigma}_{a}^{-1}(A_2^{1/p})=\zeta_p^{-j_2-j_1}A_2^{1/p}A_1^{1/p}/\Tilde{\sigma}_{a}^{-1}(\beta^{\frac{p-1}{2}})
\]
\[
\Tilde{\sigma}_{b}(A_2^{1/p})=\zeta_p^{-i_2}A_2^{1/p}
\]

\[
\Tilde{\sigma}_{A_1}(A_2^{1/p})=\Tilde{\sigma}_{a}\Tilde{\sigma}_{b}\Tilde{\sigma}_{a}^{-1}\Tilde{\sigma}_{b}^{-1}(A_2^{1/p})=\zeta_p^{i_1}A_2^{1/p}
\]
Hence $\Tilde{\sigma}_{A_1}^p=1$.
\[
\Tilde{\sigma}_{a}\Tilde{\sigma}_{A_1}\Tilde{\sigma}_{a}^{-1}\Tilde{\sigma}_{A_1}^{-1}(A_2^{1/p})=\zeta_pA_2^{1/p}
\]
Define $\Tilde{\sigma}_{A_2}=\Tilde{\sigma}_{a}\Tilde{\sigma}_{A_1}\Tilde{\sigma}_{a}^{-1}\Tilde{\sigma}_{A_1}^{-1}$. Then $\Tilde{\sigma}_{A_2}^p=1$. And one can check 
\[
\Tilde{\sigma}_{a}\Tilde{\sigma}_{A_2}\Tilde{\sigma}_{a}^{-1}\Tilde{\sigma}_{A_2}^{-1}=1,\Tilde{\sigma}_{b}\Tilde{\sigma}_{A_2}\Tilde{\sigma}_{b}^{-1}\Tilde{\sigma}_{A_2}^{-1}=1,\Tilde{\sigma}_{A_1}\Tilde{\sigma}_{A_2}\Tilde{\sigma}_{A_1}^{-1}\Tilde{\sigma}_{A_2}^{-1}=1
\]
These imply that $\Tilde{\sigma}_{a},\Tilde{\sigma}_{b} $ satisfies the relations in the definition of $M_2$. And $ \Tilde{\sigma}_{A_1}, \Tilde{\sigma}_{A_2}$ plays the role as $ t_1, t_2$. Hence we have an isomorphism $ \Gal(K(a^{1/p},b^{1/p},A_1^{1/p},A_2^{1/p})/K)\cong M_2$. It determines a proper defining system $ G_K\rightarrow \Gal(K(a^{1/p},b^{1/p},A_1^{1/p},A_2^{1/p})/K)\cong M_2\subset U_4(\F_p)\subset\Bar{U}_{5}(\F_p)$. And $Res_{\ker\chi}\psi_2=\chi_{A_2}$.

The general case can be checked by the same process and the calculation is tedious. We omit here.

\end{proof}

\begin{remark}
    The proof for the case $M_1$ is the same as the proof of theorem \ref{sharifi explicit} in \cite{MR3614934} and \cite{MR2716836}. We just imitate the proof and get the generalized result. But the calculation becomes more and more tedious. 
\end{remark}
 Now, if we have $\chi\cup\psi_0=0$, we can construct a Galois field extension and it corresponds to a proper defining system. How do we get all proper defining systems? We need the following definition and lemma.
 \begin{definition}
     The proper defining system we get by the way in lemma \ref{standard} is called the standard proper defining system.
 \end{definition}
 By previous lemmas, fix $\chi$ and $\psi_0$, the standard proper defining system depends on the choice of $a,b,\beta$ and choices of the lifting of $\sigma_a$ and $\sigma_b$.   
\begin{lemma}\label{str of proper}
    All proper defining systems with respect to $\chi$ can be obtained from the standard proper defining system by operations in lemma \ref{blocklemma}  and remark \ref{remark operation}.
\end{lemma}
\begin{proof}
    Assume we have a proper defining system $\Bar
{\rho}_n:G_K\rightarrow\Bar{U}_{n+1}(\F_p),n\leq p$. 
\[\Bar{\rho}_n=\begin{bmatrix}  
    1 & \chi &\binom{\chi}{2}&\binom{\chi}{3}&\binom{\chi}{4}&\cdots & *\\ 
    0 &1& \chi &\binom{\chi}{2}&\binom{\chi}{3}&\cdots & \psi_{n-2}\\
    \vdots& \vdots& \vdots& \vdots& \vdots& \ddots& \vdots\\
   0&0&0& 1&\chi&\binom{\chi}{2}&\psi_2\\
   0&0&0&0&1&\chi&\psi_1\\
   0&0&0&0&0&1&\psi_0\\
   0&0&0&0&0&0&1
    \end{bmatrix}
\]
Then $\chi\cup\psi_0=0 $. By lemma \ref{standard}, we can get a standard proper defining system:

\[\Bar{\rho}'_n=\begin{bmatrix}  
    1 & \chi &\binom{\chi}{2}&\binom{\chi}{3}&\binom{\chi}{4}&\cdots & *\\ 
    0 &1& \chi &\binom{\chi}{2}&\binom{\chi}{3}&\cdots & \psi_{n-2}'\\
    \vdots& \vdots& \vdots& \vdots& \vdots& \ddots& \vdots\\
   0&0&0& 1&\chi&\binom{\chi}{2}&\psi_2'\\
   0&0&0&0&1&\chi&\psi_1'\\
   0&0&0&0&0&1&\psi_0'\\
   0&0&0&0&0&0&1
    \end{bmatrix}
\]
where $\psi_0=\psi_0'$. Assume $ \psi_i=\psi_i'$ for $ 0\leq i\leq m$. Then by lemma \ref{blocklemma}
\[
\begin{bmatrix}  
    1 & \chi &\binom{\chi}{2}&\binom{\chi}{3}& 
    \binom{\chi}{4}& \cdots &\cdots & *\\ 
    0 &1& \chi &\binom{\chi}{2}&\binom{\chi}{3}& \cdots
    &\cdots & \psi_{n-2}'-\psi_{n-2}\\
    \ddots& \ddots& \ddots& \ddots& \ddots& \vdots &\vdots& \vdots\\
    \cdots & 0& 1& \chi& \binom{\chi}{2}& \ldots& \binom{\chi}{m+1}& \psi_{m+1}'-\psi_{m+1}\\
    0& \cdots& 0 & 1& \chi& \cdots& \binom{\chi}{m}& 0\\
    \vdots& \vdots& \vdots& \vdots& \vdots& \vdots &\vdots& \vdots\\
   0&0&0&0& 1&\chi&\binom{\chi}{2}&0\\
   0&0&0&0&0&1&\chi&0\\
   0&0&0&0&0&0&1&0\\
   0&0&0&0&0&0&0&1
    \end{bmatrix}
\] 
 is a proper defining system that is induced by the following defining system:
\[\begin{bmatrix}  
    1 & \chi &\binom{\chi}{2}&\binom{\chi}{3}&\binom{\chi}{4}&\cdots & *\\ 
    0 &1& \chi &\binom{\chi}{2}&\binom{\chi}{3}&\cdots & \psi_{n-2}'-\psi_{n-2}\\
    \vdots& \vdots& \vdots& \vdots& \vdots& \ddots& \vdots\\
   0&0&0& 1&\chi&\binom{\chi}{2}&\psi_{m+3}'-\psi_{m+3}\\
   0&0&0&0&1&\chi&\psi_{m+2}'-\psi_{m+2}\\
   0&0&0&0&0&1&\psi_{m+1}'-\psi_{m+1}\\
   0&0&0&0&0&0&1
    \end{bmatrix}
 \]
Then the lemma is followed by induction.

\end{proof}
In the lemma \ref{standard}, we can construct a standard proper defining system from a field extension. Conversely, we can get a field extension from a proper defining system.
\begin{lemma}
   Given a proper defining system $ \Bar{\rho}_n:G_K\rightarrow \Bar{U}_{n+1}$, the fixed subfield by $\ker\Bar{\rho}_n$ can be writen in the form  
   
 \[  K 
 \left( 
 \begin{array}{cc}
  a^{1/p},b_0^{1/p},(A_{0,1}b_1)^{1/p},(A_{0,2}A_{1,1}b_2)^{1/p},(A_{0,3}A_{1,2}A_{2,1}b_3)^{1/p},\\ \cdots, (A_{0,n-2}A_{1,n-3}\cdots A_{n-2,1}b_{n-2})^{1/p}
\end{array}  
\right)
\]
           where $a,b_0,b_1,\cdots,b_{n-2}\in K$ and there exist $\beta_i\in K(a^{1/p})$ satisfying that $\Nrm_{K(a^{1/p})/K}(\beta_i)= b_i$ for $0\leq i\leq n-3$ and $ A
_{i,j}:=\prod_{k=0}^{p-1}\sigma_a^k(\beta_i^{\binom{k}{j}})$
   
\end{lemma}
 \begin{proof}
     The result directly follows from lemma \ref{str of proper} by using the correspondence between the standard proper defining system and $M_n$-field extension.

     Another method to check the lemma is by using the same method in the proof of lemma \ref{standard}. Check that the field extension is Galois extension and $\Tilde{\sigma}_a$ and $\Tilde{\sigma}_{b_0}$ are generators of $M_n$ and satisfy the relations. The calculation will become tedious, we omit it here. 
 \end{proof}

 Now, back to our case that $K$ is a number field. Instead consider the group $G_K=\Gal(K^{sep}/K)$, we consider the group $G_{K,S}=\Gal(K^S/K)$. For a defining system $\Bar{\rho}_n: G_{K,S}\rightarrow \Bar{U}_{n+1}$, the subfield $L$ fixed by $\ker\Bar{\rho}_n$ is a field extension that is unramified outside $S$. Conversely, if we have a field extension $L/K$ that is unramified outside $S$ and $\Gal(L/K)$ is isomorphic to a subgroup of $\Bar{U}_{n+1} $, then we have a defining system $\Bar{\rho}_n: G_{K,S} \rightarrow \Gal(L/K)\subset \Bar{U}_{n+1}$. We need the following known fact.
\begin{lemma}\label{unramified extension}
  Let $K$ be a number field containing $p$-th root of the unit. Let $a\in K^*$ and $\p$ be a prime that does not divide $p$, then $K(a^{1/p})/K$ is unramified at $\p$ if and only if the valuation $v_{\p}(a)\equiv0\mod{p}$ 
\end{lemma}

Now, we apply all we have to the case \ref{case2} where $K$ is an imaginary quadratic field and $p$ splits in $K$ and the size of the class group $h_K$ is prime to $p$. Notations as the beginning of case \ref{case2} and the beginning of this subsection \ref{numerical criterion1}, recall that the character $\chi$ is $G_{K,S}\rightarrow \Gal(K_1/K)\cong\F_p$ where $K_1=K\Q_1$. 

\begin{lemma}\label{p-adic log}
    Let $\psi_0\in H^1(G_{K,S},\mu_p)\cong K^*\cap K_S^{*p}/K^{*p}$ correspond $\alpha\in K^*\cap K_S^{*p}/K^{*p} $. Then $\chi\cup\psi_0=0$ if and only if that $\log_p(\alpha)\cong 0 \mod{p^2}$ where $\log_p$ is the $p$-adic log.
\end{lemma}
\begin{proof}
 Let $p\OO_K=\p_0\Tilde{\p}_0$ and $p\OO_{K_1}=\p_1\Tilde{\p}_1$. Let $K_{\p_0}$ and $K_{\Tilde{\p}_0}$ be the completion of $K$ at prime $ \p_0$ and $\Tilde{\p}_0$ respectively.  Similarly for the definition $ K_{1,\p_1}$ and $K_{1,\Tilde{\p}_1}$. We will use Poitou-Tate Duality which is the theorem 8.6.7 in \cite{MR2392026}. We use the same notation as in chapter 8.6 of \cite{MR2392026}. Take $A=\mu_p$, then $A'=\Hom(\mu_p,\OO_S^*)\cong \F_p$ as $G_{K,S}$ module. Lemma 8.6.3 in \cite{MR2392026} tells us $\Sh^1(G_{K,S},\F_p)\cong \Hom(\Cl_S(K),\F_p)=0 $. By theorem 8.6.7 in \cite{MR2392026}, We have $\Sh^2(G_{K,S},\mu_p)\cong \Sh^1(G_{K,S},\F_p)^{\vee}=0$. This implies that the map $ H^2(G_{K,S},\mu_p)\rightarrow H^2(G_{K_{\p_0}},\mu_p)\oplus H^2(G_{K_{\Tilde{\p}_0}},\mu_p)$ is injective. The cup product $ \chi\cup\psi_0$ vanishes in $ H^2(G_{K,S},\mu_p)$ if and only if $Res\chi\cup Res\psi_0$ vanishes both in $H^2(G_{K_{\p_0}},\mu_p)$ and $ H^2(G_{K_{\Tilde{\p}_0}},\mu_p)$. In local field $K_{\p_0}$, $\chi\cup\psi_0 $ vanishes if and only if $\alpha\in \Nrm_{ K_{1,\p_1}/K_{\p_0}}( K_{1,\p_1}^*)$. One can check that $ K_{\p_0}=\Q_p$ and $K_{1,\p_1}=\Q_{1,p}$ which is the completion of $\Q_1$ at $p$. And $\Q_{1,p}/\Q_p$ is totally ramified degree $p$ extension. We can decompose $ \Q_p^*=p^{\Z}\oplus \F_p^*\oplus (1+p\Z
_p)$ since $p$ is odd and $\Z_p^*\cong \F_p^*\oplus (1+p\Z_p)\cong \F_p^*\oplus\Z_p $ is given by $ t\rightarrow (t\mod{p},\log_p(t)/\log_p(1+p))$. By local class field theory, we have $ \Nrm_{ K_{1,\p_1}/K_{\p_0}}( K_{1,\p_1}^*)\cong \Z\oplus \F_p^*\oplus (1+p^2\Z
_p)$ as a subgroup of $\Q_p^*$.  We have $\alpha \in \Nrm_{ K_{1,\p_1}/K_{\p_0}}( K_{1,\p_1}^*)$ if and only if $\log_p(\alpha)=0\mod{p^2}$. And similar story happened when we complete at $\Tilde{\p}_0$. Hence $\log_p(\alpha)=0\mod{p^2}$ if and only if $\chi\cup\psi_0=0$ 

 \end{proof}
 \begin{remark}
    The value $\log_p\alpha$ depends on the embedding $K\rightarrow \Q_p^{sep}$. But the criterion  $\log_p(\alpha)\equiv0 \bmod{p^2}$ does not depend on the embedding.
 \end{remark}

Here is the relation between our result and the classical result Gold criterion \cite{Gold}
\begin{theorem}
    Let $K$ be a imaginary quadratic field, $p\nmid h_K$, $p$ split in $K$ i.e. $\p\OO_K=\p_0\Tilde{\p_0}$. Take $(\alpha)=\p_0^{h_K}$, then $\lambda \geq 2\Leftrightarrow$ the cup product $\chi\cup \alpha=0\Leftrightarrow \log_p(\alpha)\equiv0 \bmod{p^2}\Leftrightarrow \alpha^{p-1}\equiv 1\mod{\Tilde{\p}_0^2}$ where $\log_p$ is $p$-adic log.
\end{theorem}
 \begin{proof}
    We only need to prove the last equivalent relation. Complete $K$ at $\Tilde{\p}_0$, We have $K_{\Tilde{\p}_0}\cong \Q_p$ and $\alpha \in \Z_p^*\cong \F_p^*\oplus 1+p\Z_p$. And $\log_p(\alpha)\equiv0 \bmod{p^2}\Longleftrightarrow \alpha \in \F_p^*\oplus 1+p^2\Z_p \Longleftrightarrow \alpha^{p-1}\equiv 1\mod{p^2}$ in $\Q_p \Longleftrightarrow \alpha^{p-1}\equiv 1\mod{\Tilde{\p}_0^2}$ in $K$.
 \end{proof}
Let  $\Bar{\rho}_n: G_{K,S}\rightarrow \Bar{U}_{n+1}$ be a proper defining system. 
\[\Bar{\rho}_n=\begin{bmatrix}  
    1 & \chi &\binom{\chi}{2}&\binom{\chi}{3}&\binom{\chi}{4}&\cdots & *\\ 
    0 &1& \chi &\binom{\chi}{2}&\binom{\chi}{3}&\cdots & \psi_{n-2}\\
    \vdots& \vdots& \vdots& \vdots& \vdots& \ddots& \vdots\\
   0&0&0& 1&\chi&\binom{\chi}{2}&\psi_2\\
   0&0&0&0&1&\chi&\psi_1\\
   0&0&0&0&0&1&\psi_0\\
   0&0&0&0&0&0&1
    \end{bmatrix}
\]
When restricted on $G_{K_1,S}=\ker\chi$, the cochain $\psi_i\in \mathcal{C}^1(G_{K,S},\mu_p)$ becomes a cocycle in $\mathcal{C}^1(G_{K_1,S},\mu_p)$. So $ Res_{G_{K_1,S}} \psi_i$ corresponds to a cocycle $(\sigma\rightarrow\sigma(A_i^{1/p})/A_i^{1/p})$ for some $A_i\in K_1^*\cap K_S^{*p}$ and the cocycle depends on the choice of $ A_i^{1/p}$ since we do not have $\mu_p\in K$. Assume that the Massey product relative to the proper defining system $\Bar{\rho}_n: G_{K,S}\rightarrow \Bar{U}_{n+1}$ vanishes, i.e. there exists a cochain $\psi_{n-1}\in \mathcal{C}^1(G_{K,S},\mu_p)$ fitting in the lifting $ \rho_n:G_{K,S}\rightarrow U_{n+1}$. 

\[\rho_n=\begin{bmatrix}  
    1 & \chi &\binom{\chi}{2}&\binom{\chi}{3}&\binom{\chi}{4}&\cdots & \psi_{n-1}\\ 
    0 &1& \chi &\binom{\chi}{2}&\binom{\chi}{3}&\cdots & \psi_{n-2}\\
    \vdots& \vdots& \vdots& \vdots& \vdots& \ddots& \vdots\\
   0&0&0& 1&\chi&\binom{\chi}{2}&\psi_2\\
   0&0&0&0&1&\chi&\psi_1\\
   0&0&0&0&0&1&\psi_0\\
   0&0&0&0&0&0&1
    \end{bmatrix}
\]
Similarly, $Res_{G_{K_1,S}}\psi_{n-1}$ is a cocycle and corresponds to a element $A_{n-1}\in K_1^*\cap K_S^{*p}$. Let $\Bar{\rho}_n':G_K\rightarrow G_{K,S}\xrightarrow{\Bar{\rho}_n}\Bar{U}_{n+1}$ be the composition of $G_K\rightarrow G_{K,S} $ and $\Bar{\rho}_n$. Then $\Bar{\rho}_n'$ is a proper defining system over $G_K$. Then the Massey product relative to $\Bar{\rho}_n'$ also vanishes. Then there exists a cochain $\psi_{n-1}':G_K\rightarrow \mu_p$ fitting in the cochain $ \rho_n':G_K\rightarrow U_{n+1}$.

\[\rho_n'=\begin{bmatrix}  
    1 & \chi &\binom{\chi}{2}&\binom{\chi}{3}&\binom{\chi}{4}&\cdots & \psi_{n-1}'\\ 
    0 &1& \chi &\binom{\chi}{2}&\binom{\chi}{3}&\cdots & \psi_{n-2}\\
    \vdots& \vdots& \vdots& \vdots& \vdots& \ddots& \vdots\\
   0&0&0& 1&\chi&\binom{\chi}{2}&\psi_2\\
   0&0&0&0&1&\chi&\psi_1\\
   0&0&0&0&0&1&\psi_0\\
   0&0&0&0&0&0&1
    \end{bmatrix}
\]
And $Res_{G_{K_1}}\psi_{n-1}'$ is a cocycle in $\mathcal{C}^1(G_{K_1},\mu_p)$ and corresponds to $A_{n-1}'\in K_1^*$. By definition of Massey product, we have $d(\psi_{n-1}'-\psi_{n-1})=0$. Hence $\psi_{n-1}'-\psi_{n-1} $ is a cocycle in $\mathcal{C}^1(G_{K},\mu_p)$ which corresponds to $ \sigma\rightarrow \frac{\sigma(f^{1/p})}{f^{1/p}}$ for some $ f\in K^*$. When restricting on $G_{K_1}$, we have $fA_{n-1}=A_{n-1}'$. We remark here that we have to choose the $p$-th root of $A_{n-1},A_{n-1}',f$ properly so that cocycles are compatible. A different choice of $p$-th root of $A_i$ changes the corresponding cocycle a multiple of $(\sigma\rightarrow\frac{\sigma(\zeta_p)}{\zeta_p})$. And $\chi\cup (\sigma\rightarrow\frac{\sigma(\zeta_p)}{\zeta_p})=0 $. For our case now, we care about when the Massey products vanish. Therefore, we do not need to care too much about the choice of the $p$-th root of the element. For our purpose, the key is that there exists $f\in K^*$ such that $fA_{n-1}=A_{n-1}'$ where $A_{n-1}\in K_1^*\cap K_S^{*p}$ and $A_{n-1}'\in K_1^*$. By lemma \ref{unramified extension}, we have an element $A\in  K_1^*\cap K_S^{*p}$ if and only if the valuation $v_{\p}(A)\equiv0\mod{p}$ where $\p$ does not divide $p$.

Now we combine all we have and explain how our numerical criterion works:
\begin{enumerate}
    \item The Iwasawa invariant $\lambda\geq1$, always true.
    \item The Iwasawa invariant $\lambda \geq 2$ $\Leftrightarrow \log_p\alpha \equiv 0 \mod{p^2}$, where $\log_p$ is the $p$-adic log (Reason: lemma \ref{p-adic log}).

   \item  Assume $\lambda\geq 2$ is true, then $ \chi\cup \alpha=0\Longleftrightarrow \exists \beta \in K^*_1$ s.t. $\Nrm_{K_1/K}(\beta)=\alpha$. Define $A_1'=\prod_{i=0}^{p-1}\sigma^i(\beta^i)\in K_1^*$ , where $\sigma$ is the generator of the group $G/N=\Gal(K_1/K)\cong \F_p$

   Claim: There exists $ \alpha_1 \in K^*$ s.t. $ v_{\p}(\alpha_1 A_1')\equiv 0 \mod{ p} $ for all $ \p\nmid p$, where $ v_{\p}$ is the valuation corresponding to prime ideal $\p$. (Reason: by previous argument, the difference between "correct" $A_1$ and $A_1'$ we constructed is an element in $\alpha_1\in K^*$. We want $\alpha_1 A_1'$ to be our $A_1$.)

   Then $\lambda \geq 3\Longleftrightarrow \chi\cup \alpha_1=0\Longleftrightarrow 
   \log_p \alpha_1 \equiv 0 \mod{p^2}$(Reason: Let $\psi_1',\psi_2'$ be the cochain in $\mathcal{C}^1(G_K,\mu_p)$ corresponding to $A_1'$ and $A_2'=\prod_{i=0}^{p-1}\sigma^{i}(\beta^{i(i-1)/2})$. Over $G_K$, we have $ \chi\cup\psi_1'+\binom{\chi}{2}\cup \psi_0=-d\psi_2'$. Restrict on $G_{K_{\p_0}}$ through $ G_K\rightarrow G_{K_{\p_0}}$, we have $\chi\cup\psi_1'\mid_{G_{K_{\p_0}}}+\binom{\chi}{2}\cup \psi_0\mid_{G_{K_{\p_0}}}=-d\psi_2'\mid_{G_{K_{\p_0}}}$. Let $ \psi_1\in \mathcal{C}^1(G_{K,S},\mu_p)$ correspond to the $A_1$.  Restrict to $G_{K_{\p_0}}$, the Massey product $\chi\cup\psi_1+\binom{\chi}{2}\cup \psi_0=\chi\cup\psi_1'+\binom{\chi}{2}\cup \psi_0+\chi\cup(\psi_1-\psi_1')$ vanishes in $H^2(G_{K_{\p_0}},\mu_p)$ if and only if $ \chi\cup(\psi_1-\psi_1')\mid_{G_{K_{\p_0}}}=0$,i.e. $\chi\cup \alpha_1=0$ in $H^2(G_{K_{\p_0}},\mu_p)$. A similar argument as lemma \ref{p-adic log}, we have the Massey product vanishing if and only if $\log_p(\alpha_1)\equiv0\mod{p^2}$. Notice that we can not directly use lemma \ref{p-adic log} since $\alpha_1$ may not be in $H^1(G_{K,S},\mu_p)$. However, since we have restricted $\psi_1-\psi_1'$ on $G_{K_{\p_0}}$, we can directly work in  $H^2(G_{K_{\p_0}},\mu_p)$. And similarly, restricting on $G_{K_{\Tilde{\p}_0}}$, we get the same result.  Since $H^2(G_{K,S},\mu_p)\rightarrow H^2(G_{K_{\p_0}},\mu_p)\oplus H^2(G_{K_{\Tilde{\p}_0}},\mu_p)$ is injective, we conclude the Massey product vanishing if and only if $\log_p(\alpha)\equiv 0\mod{p^2}$.)

   \item Assume $\lambda \geq 3$, then $\chi \cup \alpha_1=0\Longleftrightarrow \exists \beta_1 \in K^*_1$ s.t. $\Nrm_{K_1/K}(\beta_1)=\alpha_1$, (Reason: the cup product $\chi \cup \alpha_1$ should be viewed in $H^2(G_K,\mu_p)$). Define $A_2'=\prod_{i=0}^{p-1}\sigma^{i}(\beta^{i(i-1)/2})\in K_1^*$ and $B_1'=\prod_{i=0}^{p-1}\sigma^i(\beta_1^i)\in K_1^*$. (Reason: we want the Massey product vanishing on $G_K$ first. So we construct $A_2'$ and $B_1'$ correspond to the cocycle in this way. See lemma \ref{standard} and lemma \ref{str of proper}.)

   Claim: There exists $ \alpha_2 \in K^*$ s.t. $ v_{\p}(\alpha_2 A_2'B_1')\equiv 0 \mod{ p} $ for all $ \p\nmid p$.(Reason: similarly to the previous case, the difference between "correct" $A_2$ and $A_2'B_1'$ is an element $\alpha_2\in K^*$.)

   Then $\lambda \geq 4\Longleftrightarrow \chi\cup\alpha_2=0 \Longleftrightarrow \log_p \alpha_2 \equiv 0 \mod{p^2}$ .(Reason: similar as previous case.)

   \item continues in a similar way $\cdots$

\end{enumerate}
\begin{remark}
    The numerical criterion is not perfect, we only know the existence of $\beta_i$ and $\alpha_1$ and we do not have a logarithm to compute them.

\end{remark}

I hope the following lemma can inspire people to come up with explicit numerical criterion though I can not do it. 
\begin{lemma}
    Let $K$ be a imaginary quadratic field, $p\nmid h_K$, $p$ split in $K$ i.e. $\p\OO_K=\p_0\Tilde{\p_0}$. Let $\Bar{\rho}_n: G_{K,S}\rightarrow \Bar{U}_{n+1}$ be a defining system. Then the Massey product $(\chi_1,\chi_2,\cdots,\chi_n)_{\Bar{\rho}_n}$ relative to a defining system $\Bar{\rho}_n$ vanishes over $G_{K,S}$ if and only if it vanishes over $G_K$
\end{lemma}
\begin{proof}
We have the following commutative diagram.
 \[
\begin{tikzcd}
G_{K_{\p_0}}\arrow[r,hook] \arrow[dr,hook] & G_K \arrow[d,two heads]  \\
&  G_{K,S}
\end{tikzcd}
\]  
It induces the following diagram.
 \[
\begin{tikzcd}
H^2(G_{K,S},\mu_p)\arrow[r,hook] \arrow[d]  &H^2(G_{K_{\p_0}},\mu_p)\oplus H^2(G_{K_{\Tilde{\p}_0}},\mu_p)   \\
 H^2(G_K,\mu_p)\arrow[ur]&
\end{tikzcd}
\] 
The row is an injective map by Poitou-Tate Duality (See proof of lemma \ref{p-adic log}). The column is an inflation map. If the Massey product vanishes over $G_{K,S}$, then after the inflation map, it vanishes over $G_K$. If the Massey product vanishes over $G_K$, we restrict on local field, then it vanishes over $G_{K_{\p_0}}$ and $G_{K_{\Tilde{\p}_0}}$. Since the row map is injective, it vanishes in $G_{K,S}$.
\end{proof}
\begin{remark}
   By the following well-known exact sequence:
\[ H^1(G_K,\mu_p)\xrightarrow{\chi\cup-} H^2(G_K,\mu_p)\xrightarrow{Res} H^2(G_{K_1},\mu_p)\]
The Massey product $\sum_{i=0}^{n-1}\binom{\chi}{i}\cup\psi_i $ in our case will be zero when restrict on $ G_{K_1}$. Therefore, there exists $\chi_b\in H^1(G_K,\mu_p)$ such that $\chi\cup\chi_b=\sum_{i=0}^{n-1}\binom{\chi}{i}\cup\psi_i $. 

\end{remark}

\subsubsection{case 2}
Let $K$ be an imaginary quadratic field and $\Cl(K)[p^\infty]=\Z/p^l\Z $. Assume $p$ remains prime over $K/\Q$, i.e. $p\OO_{K}=\p_0$. Then there is only one prime that is ramified in the $\Z_p$ cyclotomic extension $K\subset K_1\subset K_2\subset \cdots \subset K_\infty$. Let $I$ be an ideal in $K$ such that $[I]\neq 0\in \Cl(K)[p]$. Let $\alpha$ be the generator of the principal ideal $I^p$. 

\begin{theorem}\label{maincase3}
    Let $K$ be an imaginary quadratic field and $\Cl(K)[p^\infty]=\Z/p^n\Z $. Assume $p$ remains prime over $K/\Q$  and $n\geq 2$. Assume $\lambda\geq n-1$, then $\lambda\geq n\Leftrightarrow$ $n$-fold Massey product $(\chi,\chi,\cdots\chi,\alpha$ is zero with respect to a proper defining system.
\end{theorem}

\begin{remark}
    It is well known that $\lambda\geq 1$ is always true in the case.
\end{remark}

\begin{proof}[proof of theorem \ref{maincase3}]
\hfill

In this case, we have $\Cl(K_l)[p^\infty]^+=0$ by Proposition 13.22 in \cite{MR1421575}. We have $D_l$ is a subgroup of $\Cl(K_l)[p^\infty]$ generated by $\p_l$. Hence $D_l=D_l^+=0$ by definition. Therefore, we have $\Cl_S(K_l)[p^\infty]=\Cl(K_l)[p^\infty]=\Cl(K_l)[p^\infty]^-$. Our case satisfies conditions in lemma \ref{no finite mudule}. We have $X\cong X_{cs}$ and $\lambda=\lambda_{cs}$. 

By the exact sequence \eqref{2},  
\[
 0\rightarrow \Cl_S(K)/p =\F_p \rightarrow H^2(G_{K,S},\mu_p)\rightarrow Br(\OO_K[1/p])[p]=0\rightarrow 0 
\]
We have $H^2(G_{K,S},\mu_p)=\F_p$. The theorem \ref{main theom} implies \[\lambda=\lambda_{cs}=\min\{n|\Psi^{(n)}\neq 0\}-1+1=\min\{n|\Psi^{(n)}\neq 0\}.\]
By the exact sequence \eqref{1}, 
\[
 0\rightarrow\OO_{K,S}^*/p\cong 
 F_p\rightarrow H^1(G_{K,S},\mu_p)\rightarrow \Cl_S(K)[p]\cong \F_p\rightarrow 0 
\]
we know that $H^1(G_{K,S},\mu_p)$ is generated by $p$ and $\alpha$. A similar argument as in theorem \ref{main1}, we know $p$ has $p$ cyclic Massey product vanishing property. Similarly, the Iwasawa invariant $\lambda \geq n \Leftrightarrow \Psi^{(i)}=0$ for all $0\leq i< n $. Assume $\Psi^{(i)}=0 $ for all $ 0\leq i< n-1$, then by theorem \ref{reduce}, $\Ima\Psi^{(n-1)}$ is generated by $\Psi^{(n-1)}([\alpha])$. So $\Psi^{(n-1)}=0$ if and only if $\Psi^{(n-1)}([\alpha])=0 $ i.e. the Massey product $(\chi^{(n-1)},\alpha)$ relative to a proper defining system vanishes by theorem \ref{LLSWW}. 
\end{proof}
\subsubsection{case 3}
Let $K$ be an imaginary quadratic field and $\Cl(K)[p^\infty]=\Z/p^l\Z $. Assume $p$ is ramified in $K/\Q$, i.e. $p\OO_K=\p_0^2$. Then there is only one prime that is ramified in the $\Z_p$ cyclotomic extension $K\subset K_1\subset K_2\subset \cdots \subset K_\infty$. Let $I$ be an ideal in $K$ such that $[I]\neq 0\in \Cl(K)[p]$. Let $\alpha$ be the generator of the principal ideal $I^p$. 

\begin{theorem}\label{main case4}
   Let $K$ be an imaginary quadratic field and $\Cl(K)[p^\infty] $ be a cyclic group and $p$ is ramified in $K/\Q$ and $n\geq 2$. Assume $\lambda\geq n-1$, then $\lambda\geq n\Leftrightarrow$ $n$-fold Massey product $(\chi,\chi,\cdots\chi,\alpha)$ is zero with respect to a proper defining system. 
\end{theorem}

\begin{remark}
    It is well known that $\lambda\geq 1$ is always true in the case.
\end{remark}

\begin{proof}[proof of theorem \ref{main case4} ]
\hfill

Similar as Theorem \ref{maincase3}, we have $\Cl(K_l)[p^\infty]^+=0$ and $D_l=D_l^+=0$ by definition. These imply $\Cl_S(K_l)[p^\infty]^- =\Cl(K_l)[p^\infty]^-=\Cl(K_l)[p^\infty]$. By the exact sequence \eqref{2},
\[
0\rightarrow \Cl_S(K)/p\cong \F_p  \rightarrow H^2(G_{K,S},\mu_p)\rightarrow Br(\OO_K[1/p])[p]=0\rightarrow 0 
\]
we have $H^2(G_{K,S},\mu_p)\cong \F_p$. Our case satisfies conditions in Theorem \ref{main theom}. The same argument as Theorem \ref{maincase3} gives us the conclusion.

\end{proof}

\subsection{Cyclotomic field}
 
Let $K=\Q(\mu_p)$ where $\mu_p$ is the group of $p$-th roots of unit as before.  Let $\omega: \Gal(\Q(\mu_p)/\Q)\cong (\Z/p\Z)^*\rightarrow \Z_p$ be the Teichm\"{u}ller character. Let $X$ be the Iwasawa module which is the inverse limit of the $p$-part of the class group of $ \Q(\mu_{p^l})$ with respect to the norm map. By Corollary 10.15 in \cite{MR1421575}, $\Gal(\Q(\mu_p)/\Q)$ acts on $X$ and $X$ is decomposed as direct sum of eigenspace, i.e. $X=\oplus_{i=0}^{p-2}\varepsilon_i X$ where $ \varepsilon_i=\frac{1}{p-1}\sum_{a=1}^{p-1}\omega^i(a)\sigma_a^{-1}\in \Z_p[\Gal(\Q(\mu_p)/\Q]$. 
Fix $i=3,5,\cdots, p-2$. Assume that $\varepsilon_i\Cl(K)[p^\infty]$ is cyclic, in other words, we have $\varepsilon_i X=\Lambda/(f_i)$. Notice that by Theorem 10.16 in \cite{MR1421575}, $\varepsilon_i\Cl(K)[p^\infty]$ is cyclic if Vandiver's conjecture holds. Let $\lambda_i= deg(f_i)$. Let $I_i$ be an ideal in $K$ such that $[I_i]\neq 0\in \varepsilon_i\Cl(K)[p]$. Let $\alpha_i$ be the lift of $[I_i]$ by the map $\varepsilon_iH^1(G_{K,S},\mu_p)\rightarrow \varepsilon_i\Cl(K)[p] $.

\begin{theorem}\label{main case5}
   Let $K=\Q(\mu_p)$. Fix $i=3,5,\cdots p-2$ and assume that $\varepsilon_i\Cl(K)[p^\infty]$ is cyclic. Let $n\geq 2$. Assume $\lambda_i\geq n-1$, then $\lambda_i\geq n\Leftrightarrow$ $n$-fold Massey product $\varepsilon_i(\chi,\chi,\cdots\chi,\alpha_i)$ is zero with respect to a proper defining system. 

\end{theorem}

\begin{proof}[proof of theorem \ref{main case5} ]
\hfill

Our case satisfies the setting up in Theorem \ref{action main theom}, where $k=\Q$, $K=\Q(\mu_p)$, $\Delta=\Gal(\Q(\mu_p)/\Q)$. We have $ \Cl(K_l)=\Cl_S(K_l)$. By the exact sequence:
\[
   0\rightarrow \varepsilon_i\Cl_S(K)/p =\F_p \rightarrow \varepsilon_i H^2(G_{K,S},\mu_p)\rightarrow \varepsilon_i Br(\OO_K[1/p])[p]=0\rightarrow 0 
\]
we have $\varepsilon_i H^2(G_{K,S},\mu_p)\cong \F_p$. By Theorem \ref{action main theom}, we have \[\lambda_i=\lambda_{i,cs}=\min\{n|\varepsilon_i\Psi^{(n)}\neq 0\}-dim_{\F_p}\varepsilon_i Br(\OO_{K}[1/p])[p]= \min\{n|\varepsilon_i\Psi^{(n)}\neq 0\}.\] 
By the exact sequence:
\[
 0\rightarrow\varepsilon_i\OO_{K,S}^*/p=0\rightarrow \varepsilon_iH^1(G_{K,S},\mu_p)\rightarrow \varepsilon_i\Cl(K)[p]=\F_p\rightarrow 0
\]
we know that $\varepsilon_iH^1(G_{K,S},\mu_p)$ is generated by $ \alpha_i$ by our definition. Assume $\lambda_i\geq n-1$, then $\Psi^{(j)}|_{\varepsilon_i H^2(G_{K,S},\Omega/I^n\otimes \mu_p)}=0 $ for $0\leq j<n-1$. So $\Ima\Psi^{(n-1)}|_{\varepsilon_i H^2(G_{K,S},\Omega/I^{n-1}\otimes \mu_p)}$ is generated by $\Psi^{(n-1)}|_{\varepsilon_i H^2(G_{K,S},\Omega/I^{n-1}\otimes \mu_p)}([\alpha_i])$. And $\Psi^{(n-1)}|_{\varepsilon_i H^2(G_{K,S},\Omega/I^{n-1}\otimes \mu_p)}([\alpha_i])=\varepsilon_i \Psi^{(n-1)}(f)$ for certain cocycle $f=\sum_{k=0}^{n-2}\psi_kx^k\in \mathcal{C}^1(G,\Omega/I^{n-1}\otimes \mu_p)$ such that $\psi_0=\alpha_i$. Hence $\varepsilon_i \Psi^{(n-1)}=0$ if and only if the $n$-fold Massey product $\varepsilon_i(\chi,\chi,\cdots\chi,\alpha_i)=0$ with respect to a proper defining system. 
\end{proof}
\begin{remark}
    This is a generalization of McCallum and Sharifi's result of proposition 4.2 in \cite{MR2019977}.
\end{remark}
\newpage

\bibliographystyle{alpha}
\bibliography{sample}

\end{document}